\documentclass[reqno,onefignum,onetabnum]{siamart171218}
\usepackage{amssymb}
\usepackage{amsmath}
\usepackage{overpic} 
\usepackage{tikz} 
\usetikzlibrary{arrows}
\usepackage{enumitem} 

\usetikzlibrary{shapes.misc}

\newtheorem{remark}{Remark}

\def\addtab#1={#1\;&=}

\def\meeq#1{\def\ccr{\\\addtab}
 \begin{align*}
 \addtab#1
 \end{align*}
  }  
  
  \def\leqaddtab#1\leq{#1\;&\leq}

\def\vc#1{\mbox{\boldmath$#1$\unboldmath}}

\def\pr(#1){\left({#1}\right)}
\def\br[#1]{\left[{#1}\right]}
\def\fbr[#1]{\!\left[{#1}\right]}
\def\set#1{\left\{{#1}\right\}}
\def\ip<#1>{\left\langle{#1}\right\rangle}
\def\iip<#1>{\left\langle\!\langle{#1}\right\rangle\!\rangle}

\def\fpr(#1){\!\pr({#1})}

\def\mapengine#1,#2.{\mapfunction{#1}\ifx\void#2\else\mapengine #2.\fi }

\def\map[#1]{\mapengine #1,\void.}

\def\mapenginesep_#1#2,#3.{\mapfunction{#2}\ifx\void#3\else#1\mapengine #3.\fi }

\def\mapsep_#1[#2]{\mapenginesep_{#1}#2,\void.}

\def\vcbr[#1]{\pr(#1)}

\def\bvect[#1,#2]{
{
\def\dots{\cdots}
\def\mapfunction##1{\ | \  ##1}
	\sopmatrix{
		 \,#1\map[#2]\,
	}
}
}

\def\vect[#1]{
{\def\dots{\ldots}
	\vcbr[{#1}]
}}

\def\vectt[#1]{
{\def\dots{\ldots}
	\vect[{#1}]^{\top}
}}

\def\Vectt[#1]{
{
\def\mapfunction##1{##1 \cr} 
\def\dots{\vdots}
	\begin{pmatrix}
		\map[#1]
	\end{pmatrix}
}}

\def\R{{\mathbb R}}

\def\E{{\rm e}}
\def\I{{\rm i}}
\def\D{{\rm d}}
\def\dx{\D x}
\def\dy{\D y}

\def\tF_#1{{\tt F}_{#1}}

\def\tFC_#1{{\tt T}_{#1}}

\def\secref#1{Section~\ref{Section:#1}}

\def\erf{{\rm erf}\,}

\def\qand{\quad\hbox{and}\quad}
\def\qqand{\qquad\hbox{and}\qquad}
\def\qfor{\quad\hbox{for}\quad}

\def\elllRpz_#1{\ell_{#1{\rm z}}^{(\lambda,R),p}}

\def\sopmatrix#1{\begin{pmatrix}#1\end{pmatrix}}

\def\PP{{\mathbb P}}

\def\VP{\Vectt[p_0(x),p_1(x),\dots]}

\def\VPP{{\mathbf P}(x,y)}
\def\VPPt{\VPP^\top}

\def\P{{P}_{n,k}^{(a,b,c)}(x,y)}
\def\Pabc{{P}_{n,k}^{(a,b,c)}(x,y) }
\def\tPabc{{\tilde P}_{n,k}^{(a,b,c)}(x,y) }
\def\ddx{{\partial \over \partial x}}
\def\ddy{{\partial \over \partial y}}
\def\ddxy{{\partial^2 \over \partial x\partial y}}
\def\ddyy{{\partial^2 \over \partial y^2}}
\def\ddxx{{\partial^2 \over \partial x^2}}

\def\PPP^(#1){{\bf P}^{(#1)}}
\def\PPPt^(#1)(#2){\PPP^(#1)(#2)^\top}
\def\FF{{\vc f}}

\title{A sparse spectral method on triangles}
\author{Sheehan Olver\thanks{Department of Mathematics, Imperial College, London, UK. (\texttt{s.olver@imperial.ac.uk})} \and Alex Townsend\thanks{Department of Mathematics, Cornell University, Ithaca, NY 14853. (\texttt{townsend@cornell.edu}) This work is supported by National Science Foundation grant No.~1645445.} \and  Geoffrey Vasil\thanks{School of Mathematics \& Statistics, The University of Sydney, Australia. (\texttt{geoffrey.vasil@sydney.edu.au})} }
\date{\today}
\begin{document}
\maketitle

\begin{abstract}
In this paper, we demonstrate that many of the computational tools for univariate orthogonal polynomials have analogues for a family of bivariate orthogonal polynomials on the triangle, including Clenshaw's algorithm and sparse differentiation operators.  This allows us to derive a practical spectral method for solving linear partial differential equations on triangles with sparse discretizations. We  can thereby rapidly solve partial differential equations using polynomials with degrees in the thousands, resulting in sparse discretizations with as many as several million degrees of freedom. 
\end{abstract} 



\def\wabc{w^{(a,b,c)}}
\def\mP{\smash{{P}_{n,k}(x,y)}}
\def\P{\smash{{P}_{n,k}^{(a,b,c)}(x,y)}}
\def\Pabc{\smash{{P}_{n,k}^{(a,b,c)}(x,y)}}
\def\tPabc{{\tilde P}_{n,k}^{(a,b,c)}(x,y) }
\def\PPabc{\smash{{P}_{n,k}^{(a,b,c)}}}
\def\tPPabc{{\tilde P}_{n,k}^{(a,b,c)}}
\def\ddx{\tfrac{\partial}{\partial x}}
\def\ddy{\tfrac{\partial}{\partial y}}
\def\ddz{\tfrac{\partial}{\partial z}}
\def\dudx{\tfrac{\partial u}{\partial x}}
\def\dudy{\tfrac{\partial u}{\partial y}}
\def\dudz{\tfrac{\partial u}{\partial z}}
\def\ddxy{\tfrac{\partial^2}{\partial x\partial y}}
\def\ddyy{\tfrac{\partial^2}{\partial y^2}}
\def\ddxx{\tfrac{\partial^2}{\partial x^2}}
\def\triangle{\lhd}

\section{Introduction}

Univariate orthogonal polynomials are fundamental in applied and computational mathematics.  They are used for the development of quadrature rules~\cite{Gautschi}, spectral theory of Jacobi operators~\cite{Teschl}, eigenvalue statistics of random matrices~\cite{Deift}, computational approximation theory~\cite{Trefethen_13}, and to derive spectral methods for the numerical solution of differential equations~\cite{Boyd,Clenshaw_57_01,Olver_13_01,Shen,Trefethen_00,Vasil_16_01}. On the contrary, multivariate orthogonal polynomials currently have a more limited impact in applications and computational methods, though it is an active research area with a promising future.

To demonstrate the potential practical importance of multivariate orthogonal polynomials, we show that many computational tools for univariate orthogonal polynomials can be generalized to a family of bivariate orthogonal polynomials on a triangle. These tools allow us to derive a sparse spectral method for solving general linear partial differential equations (PDEs) with Dirichlet and Neumann conditions on triangles. While the techniques are general, we demonstrate the method on the following PDEs:
\begin{align*}
\Delta^2 u &= f(x,y), &\hbox{(Biharmonic)} \\
u_y & = u_x, &\hbox{(Transport)} \\
\Delta u + V(x,y) u &= f(x,y). & \hbox{(Variable coefficient Helmholtz)}
\end{align*}


Since triangles can be mapped to each other by affine translations, and polynomials remain polynomial,  we can consider a single reference triangle, without loss of generality. Throughout this paper, we select the reference triangle to be the {\it unit simplex}: a right-angled triangle of unit height and width, i.e., $T = \{(x,y) : 0 < x < 1, 0 < y < 1-x\}$.  

There are several different families of bivariate orthogonal polynomials on $T$~\cite{Dunkl_14_01}. Here, we consider a family that is built from univariate orthogonal polynomials~\cite{Koornwinder_75_01}:
\begin{equation} 
\begin{aligned}
P_{n,k}(x,y)  &= \tilde P_{n-k}^{(2k+1,0)}\!(x)(1-x)^k \tilde P_k^{(0,0)}\!\left(\tfrac{y}{1-x}\right),\qquad n\geq k\geq 0,
\end{aligned}
\label{eq:Koornwinder}
\end{equation}
where  $\tilde P_k^{(a,b)}(x)$ denotes the degree $k$ shifted Jacobi polynomial on $[0,1]$ with parameters $(a,b)$.\footnote{In particular, $\tilde{P}_k^{(a,b)}(x)= P_k^{(a,b)}(2x-1)$, where $P_k^{(a,b)}$ is the degree $k$ Jacobi polynomial on $[-1,1]$ with parameters $(a,b)$.}  The polynomials in~\cref{eq:Koornwinder} are one possible generalization on triangles of the Legendre polynomials~\cite[Tab.~18.3.1]{DLMF}. In particular, the polynomials satisfy three-term recurrence relations (see~\cref{eq:bivariateRecurrence}) and are orthogonal with respect to the standard $L^2$ inner-product on $T$: 
\[
\iint_T P_{n,k}(x,y) P_{m,\ell}(x,y) \dx \dy = \begin{cases} \frac{1}{\pi_{n,k}}, & (n,k)=(m,\ell), \\ 0, & (n,k)\neq(m,\ell),\end{cases}
\]
where $\pi_{n,k} = 2(2k+1)(n+1)$.
They provide a well-conditioned basis to represent integrable functions $f\in L^2(T)$ as a series expansion, 
\[
f(x,y) \!= \!\!\sum_{n=0}^{\infty} \sum_{k=0}^n f_{n,k} P_{n,k}(x,y), \quad f_{n,k}\! =\! \pi_{n,k}\!\! \iint_{T} \!f(x,y) P_{n,k}(x,y) \dx \dy, 
\]
where the first equality above should be understood in the $L^2$-sense. In order to do efficient computations with functions defined on a triangle, it is important to be able to rapidly compute expansion coefficients of $f(x,y)$ so that 
$$f(x,y) \approx \sum_{n=0}^{N} \sum_{k=0}^n a_{n,k} P_{n,k}(x,y)$$ 
for a selected integer $N$. Recently, Slevinsky developed and implemented a fast backward stable algorithm for precisely this task~\cite{Slevinsky_17_03,Slevinsky_17_02}, accompanied with an optimized multithreaded open-source C library~\cite{FastTransforms}, allowing expansions to be computationally feasible for relatively large $N$. This has greatly improved the practicality of spectral methods for triangular domains. 

The use of bivariate orthogonal polynomials on triangles has a long history in the spectral element method and $p$-finite element method ($p$-FEM) literature~\cite{Karniadakis_Sherwin_13}, going back to Dubiner~\cite{Dubiner_91_01}. The polynomials in~\cref{eq:Koornwinder} lead to highly structured $p$-FEM discretization matrices for PDEs of the form $\mathcal{L}u = -\nabla \cdot (A(x,y) \nabla u)$, and when $A(x,y)$ is a constant one can derive sparse discretizations that can be generated in optimal complexity~\cite{Beuchler_Schoeberl_06_01}. The present work can be viewed as a generalization of this construction to strong formulations of PDEs that are not necessarily elliptic. Moreover, the properties of bivariate orthogonal polynomials allows us to retain sparsity for high differential order and variable coefficient PDEs (see~\cref{ex:Helmholtz1}).

Our main idea is to exploit a hierarchy of sparse recurrence relations~\cite{Olver_18_01} that hold between the polynomials in~\cref{eq:Koornwinder} and the so-called Jacobi polynomials on the triangle~\cite{Dunkl_14_01,Koornwinder_75_01}:\footnote{The polynomials $P_{n,k}^{(a,b,c)}$ for $a,b,c>-1$ satisfy $\iint_T P_{n,k}^{(a,b,c)}(x,y) P_{m,\ell}^{(a,b,c)}(x,y)x^ay^b(1-x-y)^c \text{d}x\text{d}y = 0$ if $n\neq m$ or $k\neq \ell$.} 
\begin{equation} 
P_{n,k}^{(a,b,c)}(x,y)  = \tilde P_{n-k}^{(2k+b+c+1,a)}\!(x)(1-x)^k \tilde P_k^{(c,b)}\!\left(\tfrac{y}{1-x}\right),\qquad n\geq k\geq 0,
\label{eq:JacobiPolynomials} 
\end{equation} 
where $a,b,c>-1$. In a manner that is analogous to the ultraspherical spectral method~\cite{Olver_13_01,Townsend_15_01}, we represent the action of partial derivatives by representing the domain and range as vectors of coefficients in different bases so that the matrix representation is sparse. For example, while $\smash{\tfrac{\partial}{\partial y} P_{n,k}}$ for $k\geq 1$ cannot be written as a sparse vector of $P_{n,k}$ coefficients, we have $\smash{\frac{\partial}{\partial y} P_{n,k} = (k+1)P_{n-1,k-1}^{(0,1,1)}}$, (see~\cref{Corollary:Differentiation}). This means that the first partial derivative with respect to $y$ has a sparse matrix representation if the range is represented as a vector of $P_{n,k}^{(0,1,1)}$ coefficients. This can be summarized as
\[
u = \sum_{n=0}^{N} \sum_{k=0}^n a_{n,k} P_{n,k}  \quad \Rightarrow \quad \frac{\partial u}{\partial y} = \sum_{n=0}^{N-1} \sum_{k=0}^n (k+1)a_{n+1,k+1} P_{n,k}^{(0,1,1)}. 
\]
Moreover, these sparse recurrence relationships form a hierarchy, in the sense that $\smash{\tfrac{\partial^s}{\partial y^s}}$ has a sparse representation if the range is represented as a vector of $\smash{P_{n,k}^{(0,s,s)}}$ coefficients, for any $s\geq 0$.
Similar, but slightly more complicated, sparse recurrence relations hold for $\smash{\tfrac{\partial^s}{\partial x^s} P_{n,k}}$ when the range is represented as vectors in $\smash{P_{n,k}^{(s,0,s)}}$ coefficients for any $s\geq 0$ (see~\cref{sec:differentiation}). 

One is also able to combine sparse representations to discretize linear PDEs. For example, the Laplacian operator $\Delta u = u_{xx} + u_{yy}$  can be represented by a sparse matrix if the range is selected to be a vector of $\smash{P_{n,k}^{(2,2,2)}}$ coefficients while the domain is a vector of $P_{n,k}$ coefficients. This is because there exist sparse conversion relationships for converting between certain $\smash{P_{n,k}^{(a,b,c)}}$ bases (see~\cref{conversion}). \Cref{fig:Laplace} illustrates a typical schema that illustrates how sparse recurrences are combined. In the language of finite-element methods, the \textit{test} and \textit{trial} spaces are different with a sparse embedding of the trial space in the test space.

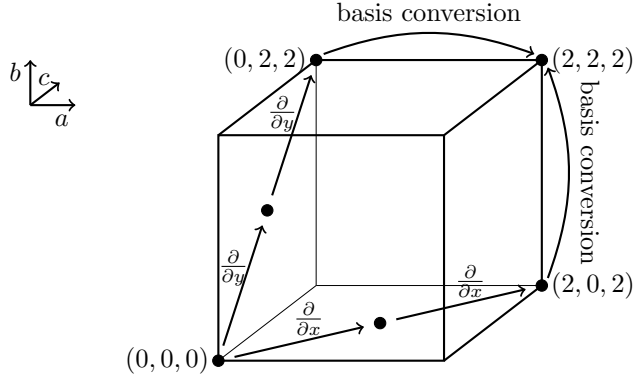
\begin{figure} 
\center
\begin{tikzpicture} 
\draw[black,solid,thick] (0,0)--(3,0)--(3,3)--(0,3)--(0,0);
\draw[black,solid,ultra thin] (1.3,1)--(4.3,1);
\draw[black,solid,thick] (4.3,1)--(4.3,4);
\draw[black,solid,thick] (4.3,4)--(1.3,4);
\draw[black,solid,ultra thin] (1.3,4)--(1.3,1);
\draw[black,solid,ultra thin] (0,0)--(1.3,1);
\draw[black,solid,thick] (3,3)--(4.3,4);
\draw[black,solid,thick] (3,0)--(4.3,1);
\draw[black,solid,thick] (0,3)--(1.3,4);
\draw[black,fill=black] (0,0) circle (.5ex);
\draw[black,fill=black] (1.5+1.3/2,1/2) circle (.5ex);
\draw[black,fill=black] (4.3,1) circle (.5ex);
\draw[black,fill=black] (1.3/2,1.5+1/2) circle (.5ex);
\draw[black,fill=black] (1.3,4) circle (.5ex);
\draw[black,fill=black] (4.3,4) circle (.5ex);
\draw[black,thick,->] (0.215,1/20)--(1.935,1/2-1/20);
\draw[] (1.2,.9) node[anchor=north] {$\tfrac{\partial}{\partial x}$};
\draw[black,thick,->] (1.5+1.3/2+0.215,1/2+1/20)--(1.5+1.3/2+1.935,1-1/20);
\draw[] (1.5+1.3/2+1.2,1/2+.9) node[anchor=north] {$\tfrac{\partial}{\partial x}$};
\draw[black,thick,->] (1.3/20,.2)--(1.3/2-1.3/20,1.8);
\draw[black,thick,->] (1.3/2+1.3/20,2.2)--(1.3/2+1.3/2-1.3/20,3.8);
\draw[] (0.2,1.6) node[anchor=north] {$\tfrac{\partial}{\partial y}$};
\draw[] (0.85,3.65) node[anchor=north] {$\tfrac{\partial}{\partial y}$};
\draw [black,thick,->] (1.4,4.08) to [out=20,in=160] (4.2,4.08);
\draw [black,thick,->] (4.38,1.1) to [out=70,in=290] (4.38,3.9);
\draw[] (4.3,4) node[anchor=west] {$(2,2,2)$};
\draw[] (4.3,1) node[anchor=west] {$(2,0,2)$};
\draw[] (0,0) node[anchor=east] {$(0,0,0)$};
\draw[] (1.3,4) node[anchor=east] {$(0,2,2)$};
\draw[black,thick,->] (-2.5,3.4)--(-2.5,4);
\draw[black,thick,->] (-2.5,3.4)--(-1.9,3.4);
\draw[black,thick,->] (-2.5,3.4)--(-2.5+1.3/2*0.6,3.4+1/2*0.6);
\draw[] (-2.3,3.2) node[anchor=west] {$a$};
\draw[] (-2.7,3.6) node[anchor=south] {$b$};
\draw[] (-2.52,3.55) node[anchor=south west] {$c$};
\draw[] (2.8,4.4) node[anchor=south] {basis conversion};
\draw[] (4.7,2.5) node[anchor=west] {\rotatebox{270}{basis conversion}};
\end{tikzpicture} 
\caption{The Laplace operator acting on vectors of $P_{n,k}=P_{n,k}^{(0,0,0)}$ coefficients has a sparse matrix representation if the range is represented as vectors of $P^{(2,2,2)}_{n,k}$ coefficients. Here, the arrows indicate that the corresponding operation has a sparse matrix representation when the domain is $\smash{P_{n,k}^{(a,b,c)}}$ coefficients, where $(a,b,c)$ is at the tail of the arrow, and the range is $\smash{P_{n,k}^{(\tilde{a},\tilde{b},\tilde{c})}}$ coefficients, where $(\tilde{a},\tilde{b},\tilde{c})$ is at the head of the arrow. 
}
\label{fig:Laplace} 
\end{figure}

The paper is organized as follows. In~\cref{Section:ops}, we establish some general computational tools for bivariate orthogonal polynomials such as Jacobi operators and the bivariate Clenshaw algorithm. In~\cref{Section:computationkoornwinder}, we specialize to~\cref{eq:JacobiPolynomials}, where the additional structure allows us to achieve a more efficient Clenshaw algorithm. In~\cref{Section:solvingpdes}, we employ weighted Jacobi polynomials on the triangle to solve PDEs such as a variable coefficient Helmholtz equation and a biharmonic equation with zero Dirichlet conditions. In~\cref{Section:dirichlet}, we extend the ideas to solve linear PDEs with nonzero Dirichlet conditions, and in~\cref{Section:systems} we demonstrate that the framework easily generalizes to systems of PDEs so that it can be used to solve the Helmholtz equation in a polygonal domain. 

The appendices contain relationships and additional formulae about orthogonal polynomials on the triangle. Our spectral method depends on explicit rational recurrence relationships that the polynomials $P_{n,k}^{(a,b,c)}(x,y)$  satisfy for  differentiation, weighted differentiation, and conversion, which we detail in \cref{Appendix:PRecurrences}. Tackling Dirichlet conditions requires a modification of the basis to enable sparse restriction operators, which we define as $Q_{n,k}^{(a,b,c)}$ in \cref{Appendix:DirichletBasis}. These also have explicit rational recurrence relationships for differentiation and conversion, which we derive in \cref{Appendix:DirichletConversion}.

\section{Computations with bivariate orthogonal polynomials} \label{Section:ops} 
In this section, we derive several computational tools for bivariate orthogonal polynomials such as the Jacobi operators, Clenshaw's algorithm, and multiplication operators. Later, in~\cref{Section:computationkoornwinder}, we specialize these tools to the Jacobi polynomials on the triangle (see~\cref{eq:JacobiPolynomials}).  

Consider a sequence of bivariate polynomials 
\[
p_{0,0}(x,y),p_{1,0}(x,y),p_{1,1}(x,y),p_{2,0}(x,y),p_{2,1}(x,y),p_{2,2}(x,y),\ldots,
\]
where $\left\{p_{n,k}\right\}_{0\leq k\leq n\leq N}$ is a basis for the space of bivariate polynomials of total degree $\leq N$,\footnote{We say that a bivariate polynomial $q(x,y)$ is of total degree $\leq N$ if $q(x,y) = \sum_{n=0}^N\sum_{k=0}^n b_{n,k} x^ky^{n-k}$ for some coefficients $b_{n,k}$.} for any integer $N$. We say that such a sequence is orthogonal with respect to a nonnegative weight function $w(x,y)$ on $\Omega \subset \mathbb{R}^2$ if 
\begin{equation} 
\iint_{\Omega} w(x,y) p_{n,k}(x,y) p_{m,\ell}(x,y) \dx \dy= \begin{cases} d_{n,k}, & (n,k) = (m,\ell), \\ 0, & (n,k)\neq (m,\ell), \end{cases} 
\label{eq:orthogonality} 
\end{equation} 
where $d_{n,k}$ are positive numbers. 

%

It is notationally convenient to write the bivariate polynomials of the same total degree as a single vector-valued polynomial~\cite{Dunkl_14_01} as follows:   
\[
\PP_n(x,y) = \Vectt[{p_{n,0}(x,y)},\dots,{p_{n,n}(x,y)}].
\]
One can then state the orthogonality condition in~\cref{eq:orthogonality} more succinctly as 
\begin{equation} 
\iint_{\Omega}w(x,y)   \PP_m(x,y) \PP_n(x,y)^\top \dx \dy = \begin{cases} D_n, & m = n, \\
              \mathbf{0}, & m \neq n,
              \end{cases}
              \label{eq:OrthogonalPolynomials} 
\end{equation}
where $D_n$ is the $(n+1) \times (n+1)$ diagonal matrix with entries $d_{n,k}$ for $0\leq k\leq n$, $\mathbf{0}$ is a matrix of all zeros of the appropriate size, and $\PP_n(x,y)^\top$ denotes the transpose of $\PP_n(x,y)$. The sequence of bivariate polynomials are normalized (orthonormal) if $D_n$ is the identity matrix for all $n\geq 0$.  We also use the notation
$$\VPP\!  = \! \vectt[\PP_0({x,y}), \PP_1({x,y}),\dots]$$
to encode all of the polynomials as a single infinite vector.

\subsection{Bivariate function approximation}
A sequence of bivariate orthogonal polynomials on $\Omega\subset \mathbb{R}^2$ can be used to approximate functions that are square integrable with respect to the associated weight function $w(x,y)$ on $\Omega$.  For example, provided $\iint_\Omega w(x,y)\left|f(x,y)\right|^2\dx\dy <\infty$, we can write
\begin{align} 
f(x,y) &= \sum_{n=0}^\infty \sum_{k=0}^n  f_{n,k} p_{n,k}(x,y) =  \sum_{n=0}^\infty \PP_n(x,y)^\top \pmb{f}_n  =  \VPPt  \vc f,
\label{eq:KoornwinderExpansions} 
\end{align} 
where $\pmb{f}_n \! = \! \vectt[f_{n,0},\dots,f_{n,n}]$\! and $\vc f\!  =\!  \vectt[\pmb{f}_0,\pmb{f}_1,\dots]$\! are the coefficients of the expansion.
Here, the first equality in~\cref{eq:KoornwinderExpansions} is understood in the sense that the difference between the left- and right-hand side is zero in the norm associated to the inner-product. 

The expansion coefficients in~\cref{eq:KoornwinderExpansions} are defined by the following integrals:  
\begin{equation} 
f_{n,k} = \frac{1}{d_{n,k}}\iint_\Omega w(x,y)f(x,y) p_{n,k}(x,y) \dx\dy, \qquad n\geq k\geq 0, 
\label{eq:coeffs} 
\end{equation} 
where $d_{n,k}$ is the orthogonality constant in~\cref{eq:orthogonality}. In practice, it is usually desirable for the expansion coefficients to rapidly decay, i.e., $\|\pmb{f}_0\|, \|\pmb{f}_1\|, \ldots$ is a rapidly decaying sequence. 


\subsection{Jacobi operators}\label{Section:JacobiOperators}
In the theory of univariate orthogonal polynomial an important object is the Jacobi operator, which is a self-adjoint linear operator given by a tridiagonal matrix~\cite{Teschl}.  It is closely related to the fact that a sequence of univariate orthogonal polynomials satisfies a three-term recurrence.  For example, if $p_0,p_1,\ldots,$ is a sequence of univariate orthogonal polynomials, then 
$$
b_{k}p_{k+1}(x) + a_k p_k(x) + c_{k-1} p_{k-1}(x)= x p_{k}(x)
$$ 
for $k\geq 1$~\cite[Thm.~3.2.1]{Szego_39_01} and 
\[
J\VP= x \VP, \qquad J = \begin{pmatrix} a_0 & b_0 \cr 
c_0 & a_1 & b_1 \cr 
&c_1 & a_2  & \ddots \cr && \ddots & \ddots\end{pmatrix}.
\]
The Jacobi operator associated with $p_0,p_1,\ldots,$ is the symmetric tridiagonal matrix obtained by a diagonal similarity transform of $J$~\cite{Teschl}. This diagonal similarity transform corresponds precisely to the normalization factors required to orthonormalize the sequence of univariate orthogonal polynomials. The transformation is possible provided $0 < b_{k}^{-1} c_{k} < \infty$ for all $k$. In particular, if $\{p_k(x)\}_{k\geq 0}$ are orthonormal, then $J$ is a symmetric tridiagonal matrix. 

A related fact that is important for designing spectral methods is that $J^\top$ can be interpreted as the ``multiplication-by-$x$" operator. That is, if $f(x) =  \VPPt \vc f$ we have
\[
x f(x) = x \VPPt \vc f = \VPPt J^\top \vc f.
\]
In other words, $J^\top \vc f$ gives the coefficients of $x f(x)$.

The analogue for bivariate orthogonal polynomials is a {\it pair} of commuting operators $J_x$ and $J_y$~\cite[\S3.4]{Dunkl_14_01}, which satisfy 
\begin{equation}
J_x \VPP= x \VPP, \qquad J_y \VPP= y\VPP.
\label{eq:bivariateRecurrence}
\end{equation} 
Here, $J_x$ and $J_y$ are block tridiagonal operators so that  
	$$J_x = \sopmatrix{A_0^x & B_0^{x} \cr C_0^x & A_1^x & B_1^{x} \cr & C_1^x & A_2^x & \ddots \cr && \ddots & \ddots}, \qquad J_y = \sopmatrix{A_0^y & B_0^{y} \cr C_0^y & A_1^y & B_1^y \cr & C_1^y & A_2^y & \ddots \cr && \ddots & \ddots},$$
where $A_n^x, A_n^y \in \R^{(n+1) \times (n+1)}$ ,  $B_n^x,B_n^y \in \R^{(n+1) \times (n+2)}$, and $C_n^x,C_n^y \in \R^{(n+2) \times (n+1)}$.  When deriving spectral methods the operators $J_x$ and $J_y$ play an important role as they can be interpreted as operators for ``multiplication-by-$x$" and ``multiplication-by-$y$," respectively. That is,
\begin{equation} 
x \VPPt \vc f = \VPPt J_x^\top \vc f \qand   y \VPPt \vc f = \VPPt J_y^\top \vc f.
\label{eq:Multiplication} 
\end{equation} 
In other words, if $f(x,y) = \VPPt \vc f$, then  $J_x^\top \vc f$ and $J_y^\top \vc f$ give the coefficients of $x f(x,y)$ and $y f(x,y)$, respectively.

\subsection{Recurrences and the Clenshaw algorithm}\label{Section:Clenshaw}
For univariate orthogonal polynomials, the three-term recurrence encoded by a Jacobi operator can be used to construct the polynomials themselves at a specified point via  forward substitution.   Clenshaw's algorithm is a closely related concept that allows the evaluation of a finite series expansion of univariate orthogonal polynomials at a point~\cite{Clenshaw_55_01}. While it is common to interpret the three-term recurrence/Clenshaw's algorithm as recursions, we prefer to interpret them as forward/backward substitution on a lower/upper triangular system associated to the Jacobi operator as this point-of-view facilitates generalization to the bivariate setting.  

Let $p_0(x), p_1(x), \ldots,$ be a sequence of univariate orthogonal polynomials such that $p_0(x) = 1$, and suppose that we wish to evaluate $f(x) = \sum_{k=0}^N a_k p_k(x)$ at $x_*\in \mathbb{R}$.
Since $p_0(x), p_1(x), \ldots,$ satisfy a three-term recurrence of the form $b_{k}p_{k+1}(x) = (x-a_k)p_{k}(x) -c_{k-1}p_{k-1}(x)$ for $k\geq 1$~\cite[Thm.~3.2.1]{Szego_39_01}, we find that 
\begin{equation} 
	L_N(x_*) \!\! \begin{pmatrix} p_0(x_*) \\[3pt]p_1(x_*) \\[3pt] p_2(x_*)\\[3pt] \vdots \\[3pt] p_N(x_*)\end{pmatrix} = \sopmatrix{1 \\[3pt] a_0 \! - \! x_* & b_0 \\[3pt] c_0 & a_1\!  -\! x_* & b_1 \\[3pt] &\ddots & \ddots & \ddots \\[3pt] & & c_{N-2} & a_{N-1}\! -\! x_* & b_{N-1} }\!\! \begin{pmatrix} p_0(x_*) \\[3pt]p_1(x_*) \\[3pt] p_2(x_*)\\[3pt] \vdots \\[3pt] p_N(x_*)\end{pmatrix}= e_0,
\label{eq:lowertriangular} 
\end{equation} 
where $b_0 p_1(x) = (a_0 - x)p_0(x)$ and $e_0 = \vectt[1,0,\dots,0]$. 
 
Forward substitution on the lower triangular linear system in~\cref{eq:lowertriangular} allows one to evaluate $p_k(x_*)$ for $k\geq 0$ from which one could evaluate $f(x_*) = \sum_{k=0}^N a_k p_k(x_*)$.  
For stability purposes, the Clenshaw algorithm evaluates expansions more directly and can be written as 
\begin{equation}\label{eq:clenshaw}
	f(x_*) = \vectt[p_0(x_*),\dots,p_N(x_*)] \pmb{a}  = e_0^\top \left( \left(L_N(x_*)\right)^{-\top} \pmb{a}\right), \qquad  \pmb{a} = \Vectt[a_0,\dots,a_N].
\end{equation}
Therefore, the Clenshaw algorithm is equivalent to solving the upper triangular linear system $(L_N(x_*))^{\top}\pmb{v} = \pmb{a}$, followed by returning the first entry of $\pmb{v}$. Since $L_N(x_*)$ only has three nonzero subdiagonals, the algorithm requires $\mathcal{O}(N)$ operations to evaluate $f(x_*) = \sum_{k=0}^N a_k p_k(x_*)$. 

%

The bivariate case is more involved. Given $(x_*,y_*)\in\mathbb{R}^2$, we would like to evaluate $f(x,y) = \sum_{n=0}^N \sum_{k=0}^n a_{n,k} p_{n,k}(x,y)$ at $(x_*,y_*)$, where (without loss of generality) we assume that $p_{0,0}(x,y) = 1$.  Since there are three-term recurrence relations in both $x$ and $y$ (see~\cref{eq:bivariateRecurrence}) we find that
\begin{equation}
	L_N(x_*,y_*) \mathbf{P}(x_*,y_*)= \sopmatrix{1 \cr A_0^x - x_*I_1 & B_0^x \cr A_0^y - y_*I_1 & B_0^y \cr C_0^x & A_1^x -x_*I_2 & B_1^x \cr 
	C_0^y & A_1^y -y_*I_2 & B_1^y \cr
	 &\ddots & \ddots & \ddots}  \mathbf{P}(x_*,y_*) = \Vectt[1,{\bf 0}_{1 \times 1}, {\bf 0}_{1 \times 1}, {\bf 0}_{2 \times 1},  {\bf 0}_{2 \times 1}, \dots],
\label{eq:bivariateL} 
\end{equation} 
where $I_m$ is the $m\times m$ identity matrix and ${\bf 0}_{m\times 1}$ is the zero vector of length $m$.  Unlike the univariate case, the system is not lower triangular and so we cannot immediately invert this system via forward recurrence to find $\mathbf{P}(x_*,y_*) $.



A reformulation that  allows for inversion is  to multiply the system to reduce the blocks above the diagonal in~\cref{eq:bivariateL} to the identity. First, note that the blocks
$$
B_n = \Vectt[B_n^x,B_n^y] \in \R^{(2n+2) \times (n+2)}
$$
have full column rank for $n\geq 0$~\cite[Theorem~3.3.4]{Dunkl_14_01}. Therefore, $B_n$ has a left-inverse $B_n^+$ for $n\geq 0$ such that $B_n^+ B_n = I_{n+2}$. It follows that an equivalent evaluation scheme can be designed from
\begin{equation} 
\tilde{L}_N(x_*,y_*) \mathbf{P}(x_*,y_*) = \Vectt[1, {\bf 0}_{1 \times 1},  {\bf 0}_{2 \times 1},\dots], \quad \tilde{L}_N(x_*,y_*) = \sopmatrix{1 \cr  & B_0^+   \cr && B_1^+ \cr &&&\ddots } L_N(x_*,y_*).
\label{eq:lowertriangularsystem} 
\end{equation} 
Since $\tilde{L}_N(x_*,y_*)$ is lower triangular we can construct $ \mathbf{P}(x_*,y_*)$ via forward substitution. 

Furthermore, a natural bivariate analogue of Clenshaw's algorithm follows from writing
$$
f(x_*,y_*) = \mathbf{P}(x_*,y_*)^\top  \vc a = {\vc e}_0^\top\! \left( \left(\tilde L_N(x_*,y_*)\right)^{-\top} \! \vc a\right).
$$
Thus $f(x_*,y_*)$ can be evaluated by solving an upper triangular linear system  using back substitution.

If $ B_n^{+}$ are dense matrices for $n\geq 0$, then forward recurrence and Clenshaw's algorithm require $\mathcal{O}(N^3)$ operations. However, in the special case of Jacobi polynomials on the triangle, the matrices involved are sparse  (see~\cref{Section:computationkoornwinder}) 
and the complexity can be reduced to $\mathcal{O}(N^2)$ operations, which is optimal.


\subsection{Multiplication operators}\label{Section:Multiplication}
The relations in~\cref{eq:Multiplication} show that $J_x^\top$ and $J_y^\top$ are operators that represent ``multiplication-by-$x$" and ``multiplication-by-$y$", respectively, in the bivariate orthogonal polynomial basis. Here, we combine these operators together to construct multiplication matrices that represent multiplication by a degree $d$ polynomial expanded as $q(x,y)= \sum_{n=0}^d \sum_{k=0}^n q_{n,k} p_{n,k}(x,y)$. 

Suppose we are given a function $f(x,y) = \sum_{n=0}^N \sum_{k=0}^n a_{n,k} p_{n,k}(x,y)$, and wish to find the expansion coefficients of $g(x,y) = q(x,y)f(x,y)$, where the degree of $f$ and $q$ can differ. Using  $J_x^\top$ and $J_y^\top$, we find that  
\[
\vc g = M_q \vc f, \qquad M_q = q(J_x^\top, J_y^\top), 
\]
where the definition of $q(J_x^\top, J_y^\top)$ is\footnote{The matrix $M_q$ is the same as that studied in the literature on bivariate functions of matrices. More precisely, $M_q$ is denoted by $M_q = q\!\left\{J_x^\top,J_y\right\}\!(I)$ in~\cite{Kressner_10_01}, where $I$ is the identity matrix and the missing transpose on the second argument is a matter of convention (see~\cite[Def. 2.1]{Kressner_10_01}).}
\begin{equation} 
q(J_x^\top, J_y^\top) = \sum_{n=0}^d \sum_{k=0}^n c_{nk} (J_x^\top)^{n-k} (J_y^\top)^k, \qquad q(x,y) = \sum_{n=0}^d \sum_{k=0}^n c_{nk} x^{n-k} y^k.
\label{eq:MonomialDefinition} 
\end{equation} 
Since $J_x$ and $J_y$ are block-tridiagonal and each matrix-matrix product increases the block-bandwidth by one, we see that $q(J_x^\top, J_y^\top)$ is also a block-banded with upper and lower block-bandwidth $d$.

The expression in~\cref{eq:MonomialDefinition} is not ideal for computations when $d$ is moderately large because of the inherent ill-conditioning in the monomial basis. It is often computationally beneficial to expand $q(x,y)$ in a bivariate orthogonal polynomial expansion and evaluate $q(J_x^\top, J_y^\top)$ using an operator-valued analogue of Clenshaw's algorithm~\cite{Slevinsky_17_01,Vasil_16_01}.   

The operator-valued analogue of Clenshaw's algorithm for evaluating $q(J_x^\top, J_y^\top)$ is equivalent to the expression:  
\begin{equation}
M_q = (e_0\otimes \mathcal{I}) (L^{-\top} \pmb q), \quad 
L = \sopmatrix{I_1\otimes \mathcal{I} \cr A_0^x\otimes \mathcal{I} - I_1\otimes J_x & B_0^x\otimes \mathcal{I} \cr A_0^y\otimes \mathcal{I} - I_1\otimes J_y & B_0^y\otimes \mathcal{I} \cr C_0^x\otimes \mathcal{I} & A_1^x\otimes \mathcal{I} - I_2\otimes J_x & B_1^x\otimes \mathcal{I} \cr 
	C_0^y\otimes \mathcal{I} & A_1^y\otimes \mathcal{I} -I_2\otimes J_y & B_1^y\otimes \mathcal{I} \cr
	 &\ddots & \ddots & \ddots} 
\label{eq:bivariateL2},
\end{equation} 
where $\mathcal{I}$ is an infinite identity matrix, $e_0$ is the first canonical unit vector, and $\pmb q$ is the vector of coefficients for $q(x,y)$ in the bivariate orthogonal polynomial expansion. Here we use the Kronecker product denoted $\otimes$.

%
In general, $J_x$ and $J_y$ have dense blocks so that the total number of nonzero entries in the principal $N\times N$ block matrix of $q(J_x^\top, J_y^\top)$ is $\mathcal{O}(N^3)$, and the complexity of constructing $q(J_x^\top, J_y^\top)$ using the operator-valued Clenshaw's algorithm is $\mathcal{O}(N^4)$ (where the total number of unknowns is $\mathcal{O}(N^2)$). In the case of Jacobi polynomials on the triangle, the blocks of $J_x$ and $J_y$ are tridiagonal and there are only $\mathcal{O}(N^2)$ nonzero entries, which can be calculated in optimal complexity.


\section{Computing with Jacobi polynomials on the triangle}\label{Section:computationkoornwinder}
We now specialize the algorithmic ideas in~\cref{Section:ops} to Jacobi polynomials on the triangle (see~\cref{eq:JacobiPolynomials}). Since these polynomials have additional structure, more efficient algorithms can be designed.

We denote the Jacobi polynomials on the triangle that are orthogonal with respect to $x^a y^b (1-x-y)^c$ with $a,b,c>-1$ by
\[
\PP_n^{(a,b,c)}(x,y) = \begin{pmatrix}P_{n,0}^{(a,b,c)}(x,y) \cr \vdots \cr P_{n,n}^{(a,b,c)}(x,y) \end{pmatrix}, \qquad  \mathbf{P}^{(a,b,c)}(x,y) = \Vectt[\PP_0^{(a,b,c)}(x), \PP_1^{(a,b,c)}(x),\dots],
\]
and note that series expansions in the $P_{n,k}^{(a,b,c)}$ basis can be expressed as 
\[
f(x,y) = \sum_{n=0}^\infty \sum_{k=0}^n f_{n,k} P_{n,k}^{(a,b,c)}(x,y) = \mathbf{P}^{(a,b,c)}(x,y)^\top \pmb{f},
\]
where $\pmb{f}$ is the vector of $P_{n,k}^{(a,b,c)}$ coefficients for $f$. These expansion coefficients can be efficiently computed from samples of $f$ by a fast, backward stable algorithm~\cite{Slevinsky_17_03,Slevinsky_17_02,FastTransforms}. We further denote the  shifted Jacobi polynomials on the unit interval as
$$
{\mathbf P}^{(a,b)}(x) = \Vectt[\tilde P_0^{(b,a)}(x), \tilde P_1^{(b,a)}(x),\dots].
$$
where the alternative ordering of $a$ and $b$ helps to build analogies with the triangle case.

%


\subsection{Conversion operators}\label{conversion}
An important property of Jacobi polynomials on the interval is that they have banded conversion operators, which translate between coefficients from expansion in $\smash{P_n^{(a,b)}}$ to $\smash{P_n^{(a+1,b)}}$ or $\smash{P_n^{(a,b+1)}}$. In terms of converting expansions between bases, we can express such conversions as
$$
f(x) =  \PPPt^(a,b)(x) \vc f =  \mathbf{P}^{(a+1,b)}(x)^\top S_{(a,b)}^{(a+1,b)} \vc f =  \mathbf{P}^{(a,b+1)}(x)^\top S_{(a,b)}^{(a,b+1)} \vc f,
$$
where $S_{(a,b)}^{(a+1,b)}$ and $S_{(a,b)}^{(a,b+1)}$ are upper bidiagonal operators, with rational entries as given in \cite[(18.9.5)]{DLMF}. 

Jacobi polynomials on the triangle have a similar property: we can  increment either $a$, $b$, or $c$ in the expansion by one: 
\meeq{
f(x,y) =   \PPPt^(a,b,c)(x,y) \FF 
	=  \PPPt^(a+1,b,c)(x,y)  S_{(a,b,c)}^{(a+1,b,c)}  \FF \ccr
		=  \PPPt^(a,b+1,c)(x,y) S_{(a,b,c)}^{(a,b+1,c)}  \FF
			=  \PPPt^(a,b,c+1)(x,y)  S_{(a,b,c)}^{(a,b,c+1)}  \FF.
	}
Each of these operators are sparse: they have block-bandwidths $(0,1)$ with diagonal blocks for $S_{(a,b,c)}^{(a+1,b,c)}$ and upper bidiagonal blocks for $S_{(a,b,c)}^{(a,b+1,c)}$ and $S_{(a,b,c)}^{(a,b,c+1)}$. The entries are rational, and can be determined in closed form by the recurrence relationships in \cref{Corollary:Conversion}. 

\subsection{Constructing Jacobi operators} 
For Jacobi polynomials, the recurrence relationships that give rise to tridiagonal Jacobi operators, representing multiplication by $x$,  are well-known. However, the Jacobi operators can alternatively be derived via lower bidiagonal {\it lowering operators} $L_{(a,b)}^{(a-1,b)}$ and $L_{(a,b)}^{(a,b-1)}$ \cite[18.9.6]{DLMF} that represent multiplication by $x$ and $1-x$:
\meeq{
x f(x) =  \PPPt^(a-1,b)(x) L_{(a,b)}^{(a-1,b)} \vc f 
=   \PPPt^(a,b)(x) S_{(a-1,b)}^{(a,b)} L_{(a,b)}^{(a-1,b)}  \vc f.
}
Similarly, $1-x$ is equivalent to $S_{(a,b-1)}^{(a,b)} L_{(a,b)}^{(a,b-1)}$. In other words, the Jacobi operator corresponding to multiplication by $x$  can be constructed via
$$
J^\top \equiv  S_{(a-1,b)}^{(a,b)} L_{(a,b)}^{(a-1,b)} \equiv I -S_{(a,b-1)}^{(a,b)} L_{(a,b)}^{(a,b-1)}  .
$$
Note that the product of a lower bidiagonal operator $L_{(a,b)}^{(a-1,b)}$ and an upper bidiagonal operator $S_{(a-1,b)}^{(a,b)} $ is a tridiagonal operator, as expected.

To construct the Jacobi operators $J_x$ and $J_y$ for Jacobi polynomials on the triangle, we first note that there exists three lowering operators that satisfy:
\meeq{
x f(x,y) = \PPPt^(a-1,b,c)(x,y) L_{(a,b,c)}^{(a-1,b,c)} \FF,\ccr
y f(x,y) =  \PPPt^(a,b-1,c)(x,y) L_{(a,b,c)}^{(a,b-1,c)} \FF,\ccr
z f(x,y) =  \PPPt^(a,b,c-1)(x,y) L_{(a,b,c)}^{(a,b,c-1)} \FF,
}
where  $z := 1- x-y$. We will use other indices to indicate multiple lowering in a row, e.g.,
$$
L_{(1,1,1)}^{(0,0,0)} := L_{(1,0,0)}^{(0,0,0)} L_{(1,1,0)}^{(1,0,0)} L_{(1,1,1)}^{(1,1,0)}
$$
corresponds to multiplication by $x y z$, where the choice for navigating the parameter tree is arbitrary.

We can construct the Jacobi operators from the lowering operators via
\[
J_x^\top = S_{(a-1,b,c)}^{(a,b,c)} L_{(a,b,c)}^{(a-1,b,c)}, \qquad
J_y^\top = S_{(a,b-1,c)}^{(a,b,c)} L_{(a,b,c)}^{(a,b-1,c)}.
\]
Note that the entries of the lowering operators can be determined by the recurrences in \cref{Corollary:Lowering}, and they are sparse. In particular, they have block-bandwidths $(1,0)$ and diagonal blocks for $L_{(a,b,c)}^{(a-1,b,c)}$ and lower-bidiagonal blocks for $L_{(a,b,c)}^{(a,b-1,c)}$ and $L_{(a,b,c)}^{(a,b,c-1)}$. This block structure ensures that $J_x$ is block-tridiagonal with diagonal block and that $J_y$ is block-tridiagonal with tridiagonal blocks.  Finally, $J_{x}$ and $J_{y}$ commute because the $L$ and $S$ operators commute:
\begin{equation}
S_{(a-1,b)}^{(a,b)} L_{(a,b)}^{(a-1,b)} = L_{(a,b)}^{(a+1,b)} S_{(a+1,b)}^{(a,b)}, \quad S_{(a,b-1)}^{(a,b)} L_{(a,b)}^{(a,b-1)} = L_{(a,b)}^{(a,b+1)} S_{(a,b+1)}^{(a,b)}.
\end{equation}

\subsection{Implementation of Clenshaw's algorithm and multiplication operators}
We now exploit the sparsity structure of $J_x$ and $J_y$ to get an $\mathcal{O}(N)$ complexity Clenshaw algorithm. In particular,  using the notation of \cref{Section:Clenshaw}, since $B_n^x$ is diagonal and $B_n^y$ is tridiagonal we can construct a  simple left-inverse $B_n^+$.  That is, we have the following structure:
	$$B_n = \Vectt[B_n^x,B_n^y] = \sopmatrix{\times \cr & \times \cr && \ddots \cr &&& \times \cr &&&&\times & 0 \cr
		\times &\times \cr\times  & \times & \times \cr &\ddots & \ddots & \ddots \cr && \times & \times & \times\cr &&&\times & \times & \times
	}. 
	$$
Let $B_1 = B_n^x[0\!:\!n,0\!:\!n]$ denote the first $(n+1) \times (n+1)$ sub-block of $B_n^x$, let $b_2 = B_n^y[n,n+1]$, and let
\meeq{
\vc b_1^\top =- { b_2^{-1}} \vect[ {\bf 0}_{1\times n-2},  {B_n^y[n,n-1] \over  B_n^x[n-1,n-1]} , {B_n^y[n,n] \over  B_n^x[n,n]}] .
}
  Then, the following matrix is a pseudo-inverse of $B_n$:
$$
B_n^+ := \begin{pmatrix}  B_1^{-1} &  {\bf 0}_{n \times n-1}  &  {\bf 0}_{n \times 1} \cr
				\vc b_1^\top &  {\bf 0}_{1 \times n-1}  &  b_2^{-1}
							\end{pmatrix}.
$$	
Note that $B_n^+$ can be applied to a vector in $\mathcal{O}(n)$ operations. When incorporation into Clenshaw's algorithm described in \secref{Clenshaw}, this gives an optimal $\mathcal{O}(N^2)$ algorithm for evaluating functions. Furthermore, when incorporated into the construction of the multiplication operators (see \secref{Multiplication}), we find that one can construct multiplication operators in $\mathcal{O}(N^2)$ operations.
	
%

\subsection{Differentiation}\label{sec:differentiation} 
Jacobi polynomials on the interval have banded recurrence relationships for their derivatives by incrementing both of the parameters, that is, we can represent
$$
f'(x) = \PPPt^(a+1,b+1)(x) D_{(a,b)}^{(a+1,b+1)} \vc f,
$$
where $D_{(a,b)}^{(a+1,b+1)} $ is zero except for the first super-diagonal~\cite[18.9.15]{DLMF}. They also have banded recurrence relationship for their weighted derivatives that decrement the parameters:
$$
{\D \over \D x} [ x^a (1-x)^b f(x)] =  x^{a-1} (1-x)^{b-1} \PPPt^(a-1,b-1)(x) W_{(a,b)}^{(a-1,b-1)} \Vectt[f_0,f_1,\dots],
$$
where $W_{(a,b)}^{(a-1,b-1)} $ is zero except for the first sub-diagonal \cite[18.9.16]{DLMF}.

These properties translate to partial derivatives of Jacobi polynomials on the triangle. That is, we have\footnote{We  have similar relationships for ${\partial \over \partial z} := {\partial \over \partial x} -  {\partial \over \partial y}$, but we omit these for brevity as they are not needed.}
	\meeq{
{\partial f \over \partial x} =\PPPt^(a+1,b,c+1)(x,y) D_{x,(a,b,c)}^{(a+1,b,c+1)} \FF, \ccr
{\partial f \over \partial y} = \PPPt^(a,b+1,c+1)(x,y) D_{y,(a,b,c)}^{(a,b+1,c+1)} \FF,
}
where the entries are derived in \cref{Corollary:Differentiation}. Both $D_{x,(a,b,c)}^{(a+1,b,c+1)}$ and $D_{y,(a,b,c)}^{(a,b+1,c+1)}$ are sparse: they are block super-diagonal, and their blocks are upper bi-diagonal and super-diagonal, respectively. Similarly, for weighted differentiation we have
	\meeq{
\ddx[x^a y^b z^c f(x,y)] = x^{a-1} y^{b} z^{c-1} \PPPt^(a-1,b,c-1)(x,y) W_{x,(a,b,c)}^{(a-1,b,c-1)}\FF,  \ccr
\ddy[x^a y^b z^c f(x,y)] = x^{a} y^{b-1} z^{c-1}  \PPPt^(a,b-1,c-1)(x,y) W_{y,(a,b,c)}^{(a,b-1,c-1)} \FF,
}
where the entries are derived in \cref{Corollary:WeightedDifferentiation}. Both $W_{x,(a,b,c)}^{(a-1,b,c-1)}$ and $W_{y,(a,b,c)}^{(a,b-1,c-1)}$ are also sparse matrices as they are block sub-diagonal, and their blocks are lower bi-diagonal and sub-diagonal, respectively. 

Combining differentiation and conversion appropriately allows us to represent more complicated differential operators. For example, the Laplacian can be expressed as an operator that takes coefficients in an $\PPP^(0,0,0)$ expansion to coefficients in an $\PPP^(2,2,2)$ expansion as follows:
$$
\Delta_{(0,0,0)}^{(2,2,2)} := S_{(2,1,2)}^{(2,2,2)} S_{(2,0,2)}^{(2,1,2)}D_{x,(1,0,1)}^{(2,0,2)} D_{x,(0,0,0)}^{(1,0,1)} + S_{(1,2,2)}^{(2,2,2)} S_{(0,2,2)}^{(1,2,2)}D_{y,(0,1,1)}^{(0,2,2)} D_{y,(0,0,0)}^{(0,1,1)}
$$
A simple calculation determines that this is also a sparse operator with block-band\-widths $(2,4)$ and blocks with bandwidths $(0,4)$, see the left figure in \cref{fig:Spy}.

Similarly, we can express the Laplacian as an operator from coefficients in an $x y (1-x-y) \PPP^(1,1,1)(x,y)$ expansion to coefficients in an $\PPP^(1,1,1)(x,y)$ expansion by using weighted derivatives and lowering operators:
\begin{equation} 
\Delta_W :=  S_{(1,0,1)}^{(1,1,1)}D_{x,(0,0,0)}^{(1,0,1)} L_{(0,1,0)}^{(0,0,0)} W_{x,(1,1,1)}^{(0,1,0)} + S_{(0,1,1)}^{(1,1,1)}D_{y,(0,0,0)}^{(0,1,1)} L_{(1,0,0)}^{(0,0,0)} W_{y,(1,1,1)}^{(1,0,0)}.
\label{eq:LapW}
\end{equation} 
This is  a sparse operator with block-bandwidths $(1,2)$ and blocks with bandwidths $(2,2)$, see the middle figure in \cref{fig:Spy}.

Finally, variable coefficients can be constructed by combining lowering, conversion, and Jacobi operators. For example, the variable Helmholtz operator
$
\Delta + v(x,y)
$
can be represented as
$$
\Delta_W + S_{(0,0,0)}^{(1,1,1)} v(J_x^\top, J_y^\top) L_{(1,1,1)}^{(0,0,0)}
$$
This still leads to a sparse discretisation, where the block bandwidths depend on the degree of $v$, see the right figure in \cref{fig:Spy} for an example with $v(x,y) + x y^2$.

\begin{figure}[tb]
\begin{center}
\begin{tabular}{ccc}
\includegraphics[width=0.3\textwidth]{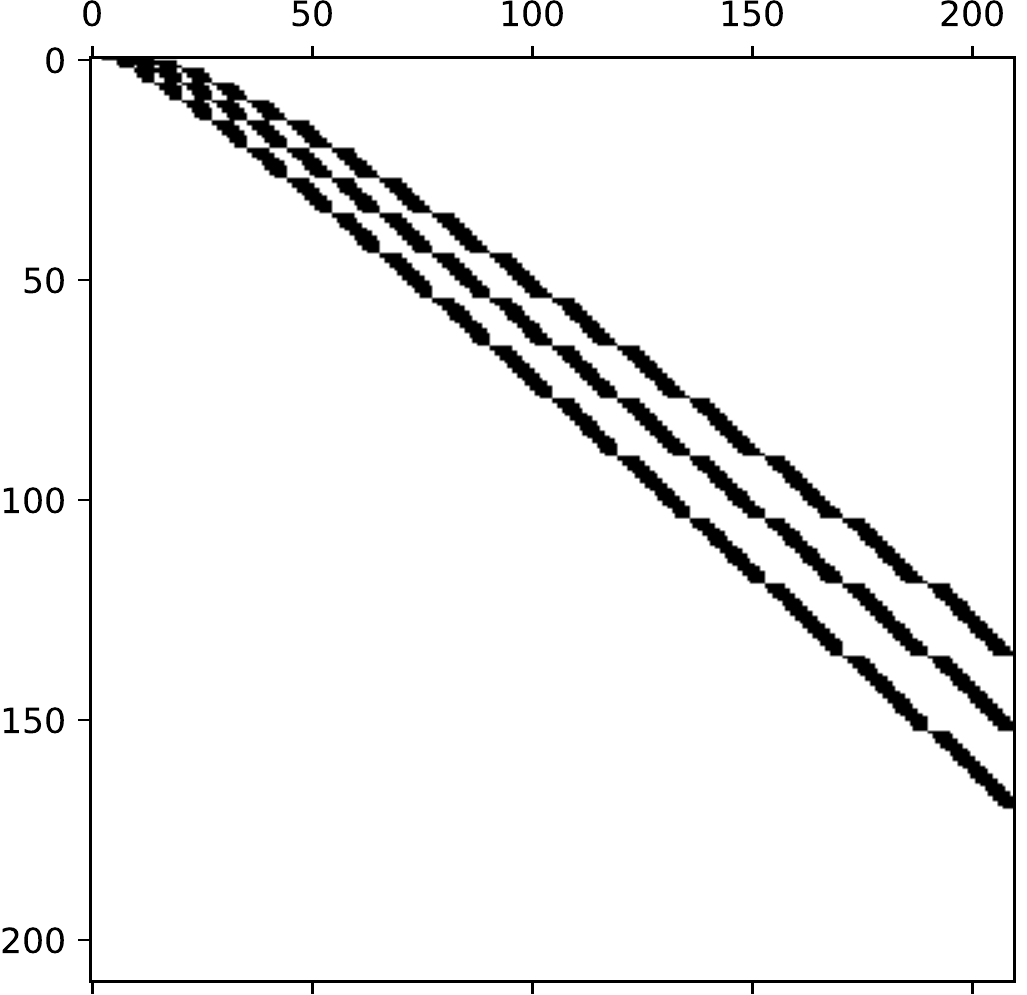}&
\includegraphics[width=0.3\textwidth]{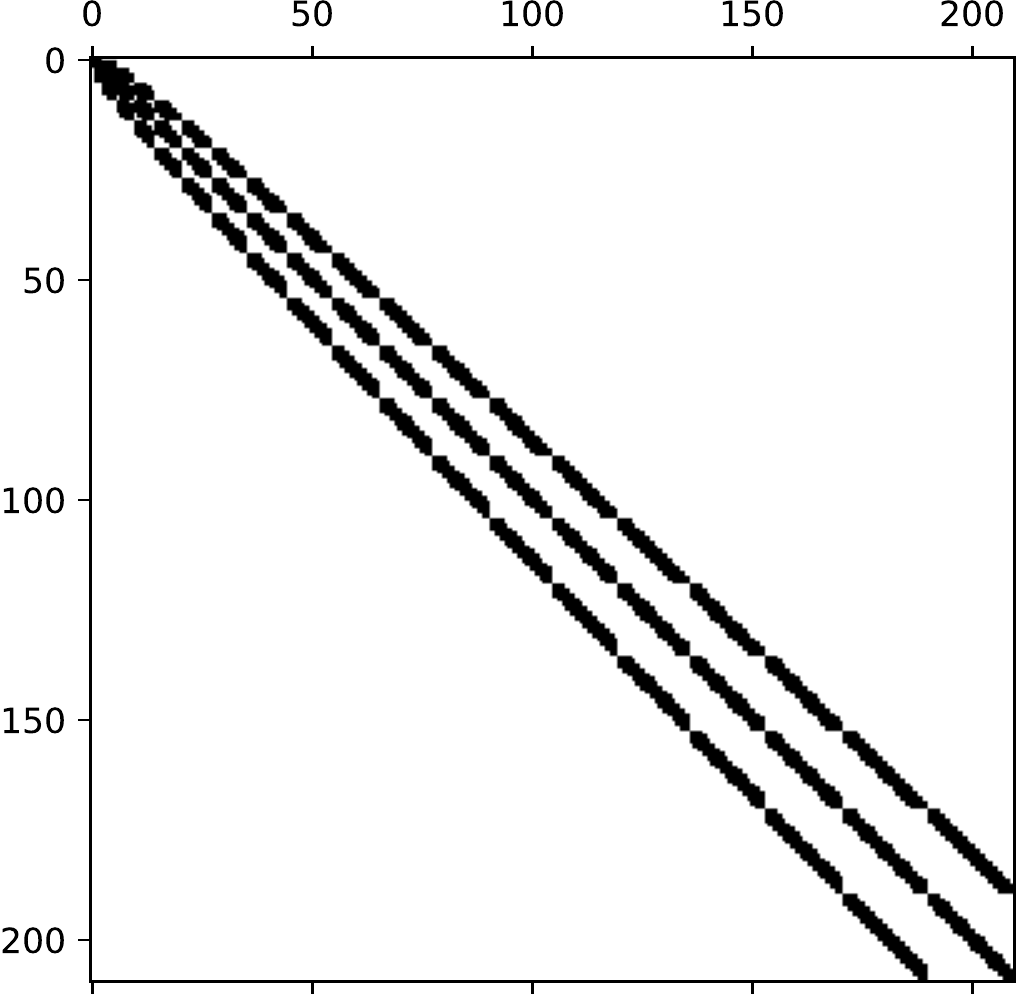}&
\includegraphics[width=0.3\textwidth]{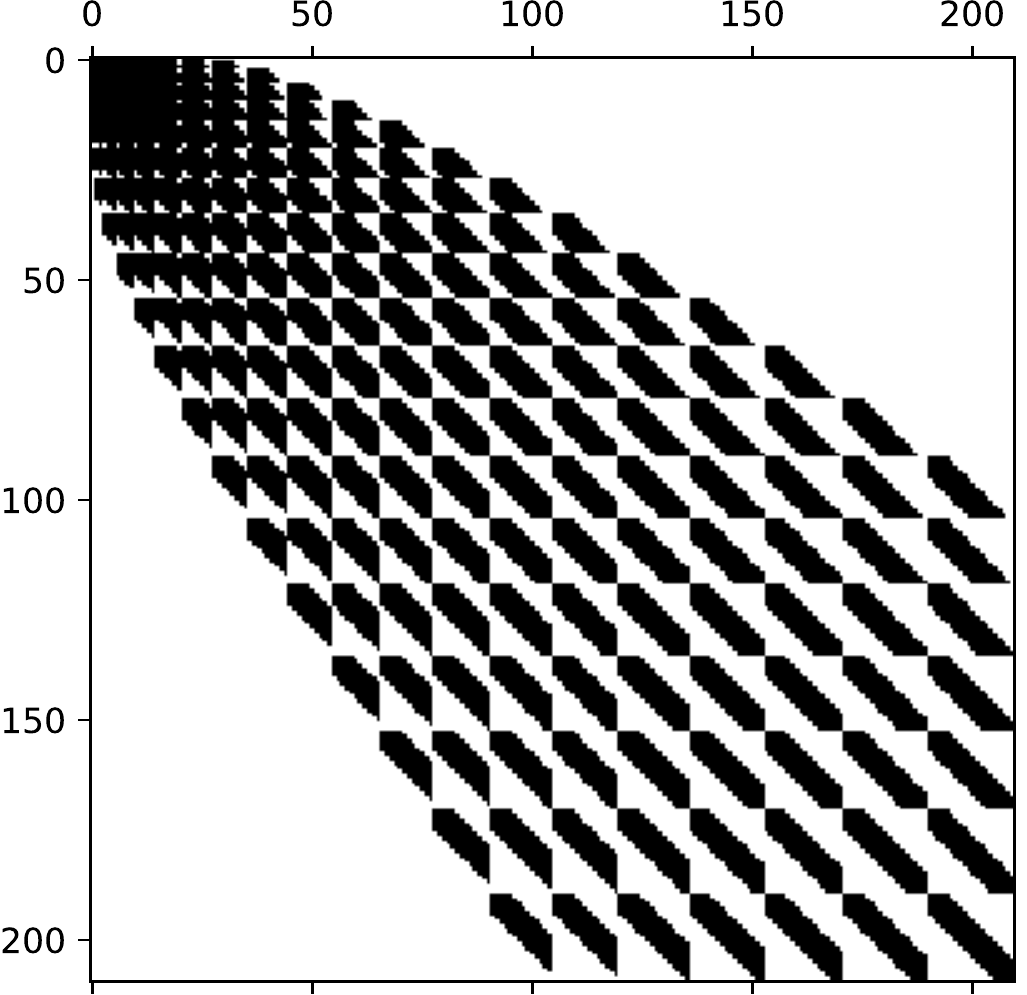}
\end{tabular}
\caption{The sparsity pattern of the Laplacian $\Delta_{(0,0,0)}^{(2,2,2)}$ (left), the weighted Laplacian $\Delta_W$ (middle), and the weighted variable coefficient Helmoltz operator   $\Delta_W + S_{(0,0,0)}^{(1,1,1)} v(J_x^\top, J_y^\top) L_{(1,1,1)}^{(0,0,0)}$ with $V(x,y) = x y^2$ (right).}
\label{fig:Spy}
\end{center}
\end{figure}

\section{Solving linear PDEs with zero Dirichlet conditions}\label{Section:solvingpdes}
We now use the systematic approach to constructing sparse operators to solve PDEs.  We construct the operators using {BlockBandedMatrices.jl} \cite{BlockBandedMatrices}, which enables fast multiplication of block-banded matrices with banded blocks by building on BLAS. We then convert the representation to a SuiteSparse compatible sparse matrix format, for matrix factorization and solves. 

\subsection{Zero Dirichlet conditions}
To solve PDEs with vanishing Dirichlet conditions, we use the weighted basis
$$x y (1-x-y) \PPP^(1,1,1)(x,y).$$
For higher order equations like the Biharmonic equation we consider vanishing Dirichlet and Neumann conditions using the weighted basis
$$x^2 y^2 (1-x-y)^2 \PPP^(2,2,2)(x,y).$$
%


\subsubsection{Example 1: Poisson equation}
Consider Poisson's equation on a triangle with zero Dirichlet conditions, i.e., 
$$
u_{xx} + u_{yy} = f(x,y), \quad (x,y) \in T, \qquad u|_{\delta T} = 0.
$$
We reduce this equation to a truncation of
$$
\Delta_w  \vc u = \vc f
$$
where the coefficients of $f(x,y) = \PPP^(1,1,1)(x,y)^\top \vc f$ are determined using \cite{Slevinsky_17_03} as implemented in \cite{FastTransforms}. 
In~\cref{fig:Poisson} (left), we depict the solution for a specific choice of $f(x,y)$. In Figure~\ref{fig:Poisson} (right), we show the construction time\footnote{Timings are performed on an iMac 2017 with 3.8 GHz Intel Core i5, using Julia v1.0 compiled with MKL BLAS.  Note that the default OpenBLAS is slower for banded matrix operations.} of the matrix (using BlockBandedMatrices.jl), execution time for an LU factorization, and  thesolve time (using SuiteSparse via Julia's SparseArrays.jl). We observe that the construction requires an optimal $\mathcal{O}(N^2)$ operations while the factorization and solution time are observed to cost an almost-optimal $\mathcal{O}(N^3)$ operations. The same complexities are observed for PDEs on rectangles using a Chebyshev-based spectral method~\cite{Julien_09_01}. 

In~\cref{fig:PoissonConvergence}, we show  the norms of each block of calculated coefficients of the approximation  for four right-hand sides with $N = 999$. Note that the rate of decay in the coefficients is a proxy for the rate of convergence of the computed solution.  The behavior of the right-hand side at the corners has an impact on the convergence rate; in particular, if $f$ and its derivatives vanish at the corners then we observe  faster convergence of the solution.  The behavior at the origin is particularly important as the Laplacian of the basis $x y (1-x-y) P_{n,k}^{(1,1,1)}(x,y)$ always vanishes at the origin. While we only observe algebraic convergence for the first three examples (that is, we do not achieve spectral convergence as $N \rightarrow \infty$), the rate of convergence is fairly fast, achieving machine precision accuracy when $f(x,y)$ vanishes at the origin  with around 10,000 unknowns.  Furthermore, the last example shows spectral convergence for a Gaussian bump function, which up to machine precision vanishes to all orders at the corners. Finally, over-resolving the solution does not result in the error plateauxing at machine precision, which means our discretization slightly improves the regularity of the data, similar to the ODE case in \cite{Olver_13_01}. 

\begin{figure}[tb]
\begin{center}
\begin{tabular}{cc}
\includegraphics[width=0.45\textwidth]{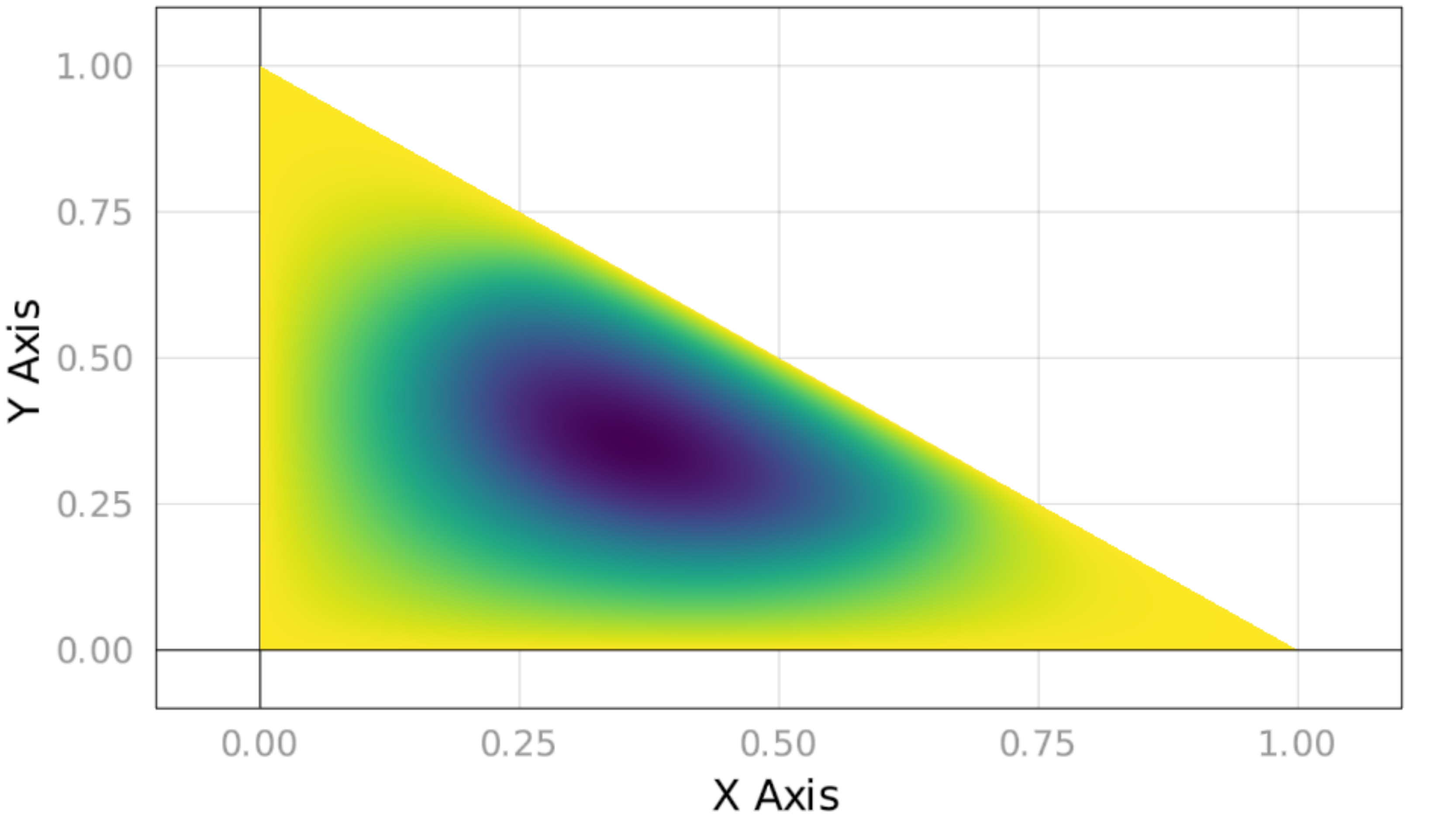}&
\includegraphics[width=0.45\textwidth]{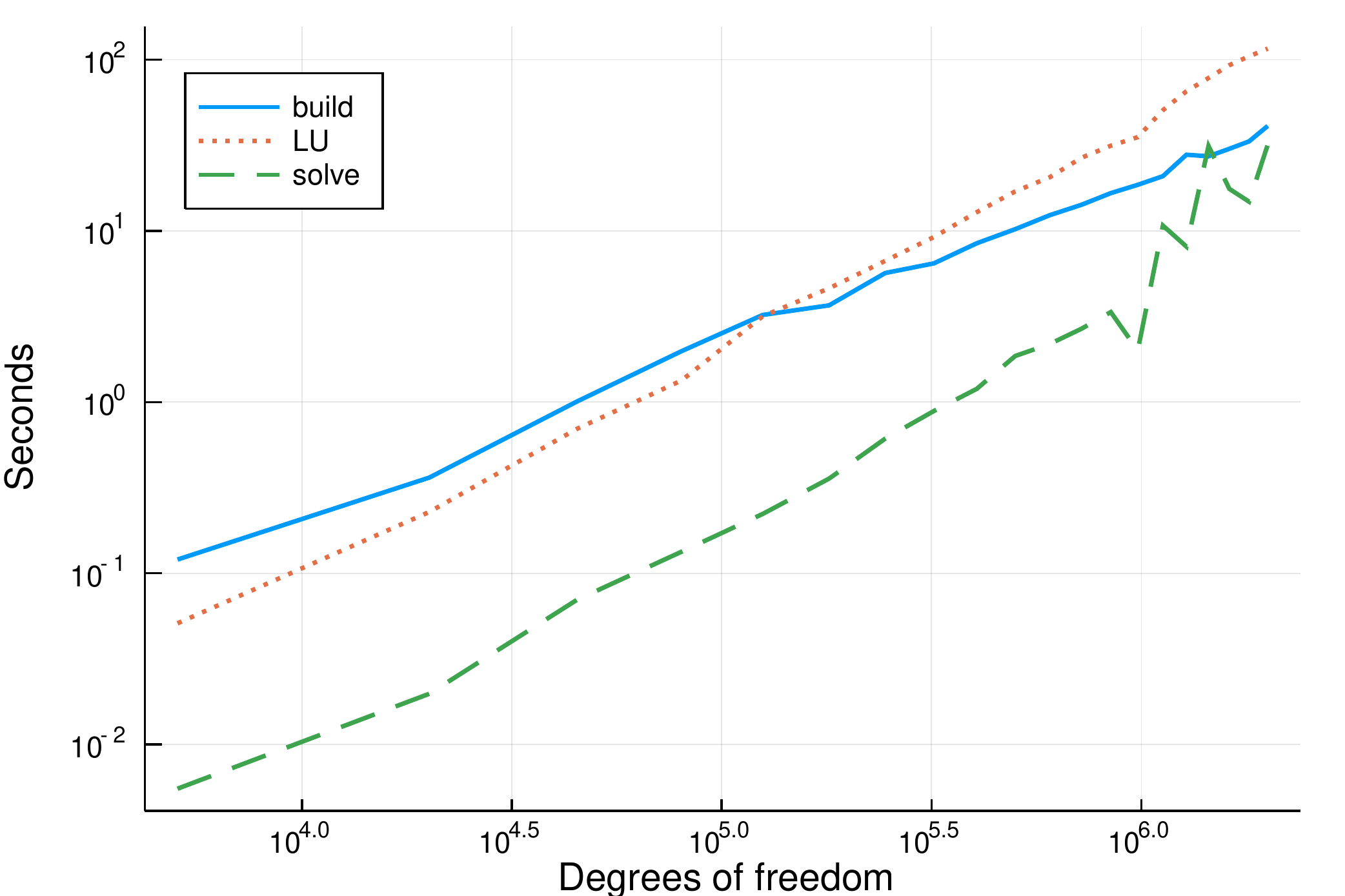}
\end{tabular}
\caption{Left: The computed solution to $\Delta u = f$ with zero boundary conditions and $f(x,y) = 1 + \erf(5 (1-10((x-1/2)^2 + (y-1/2)^2))$. Right: the time in seconds to build the discretization, calculate its LU factorization using SuiteSparse, and solve the system.}
\label{fig:Poisson}
\end{center}
\end{figure}

\begin{figure}[tb]
\begin{center}
\begin{tabular}{cc}
\includegraphics[width=0.7\textwidth]{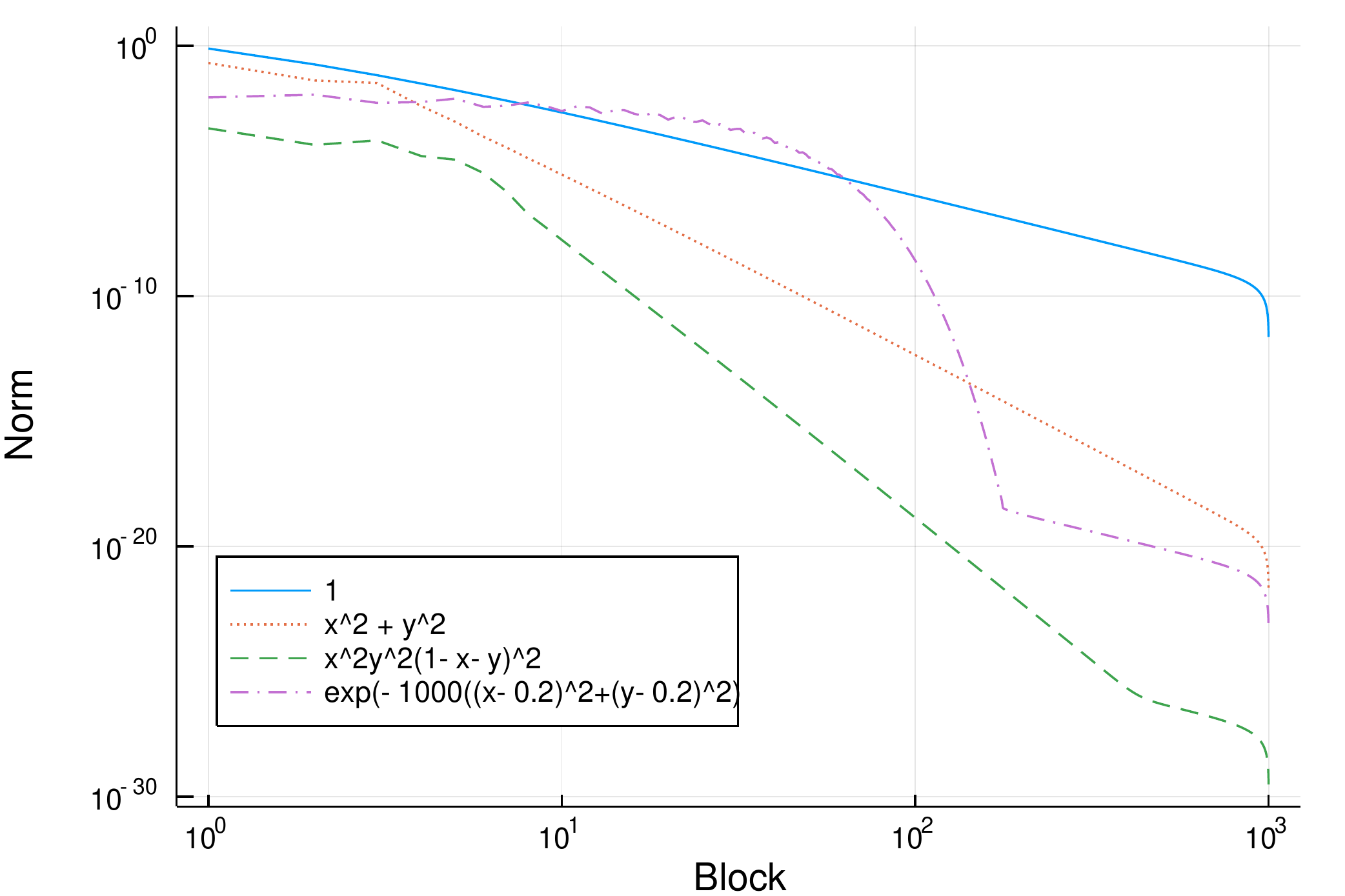}
\end{tabular}
\caption{The norm of the blocks of the calculated coefficients for four functions, for $N = 1000$, i.e., with  $~500k$ degrees of freedom. The rate in decay serves as a proxy for the error in the computed solution. We see the first three cases we have algebraic convergence, with the convergence rate improving when the function vanishes to higher order at the corners. The last example shows spectral convergence for a Gaussian bump. 
}
\label{fig:PoissonConvergence}
\end{center}
\end{figure}

\subsubsection{Example 2: Variable coefficient Helmholtz equation with forcing terms}\label{ex:Helmholtz1} 
Now, consider a variable coefficient Helmholtz equation with zero Dirichlet conditions, i.e.,
$$
u_{xx} + u_{yy} + k^2 v(x,y) u = x y \E^x, \quad (x,y) \in T, \qquad u|_{\delta T} = 0. 
$$
We first approximate $v(x,y)$ by a polynomial \cite{FastTransforms} and then use the operator-valued Clenshaw's algorithm to construct $v(J_x^\top, J_y^\top)$. We obtain the following discretization:
$$
\Delta_W + k^2.  S_{(0,0,0)}^{(1,1,1)}  v(J_x^\top, J_y^\top)  L_{(1,1,1)}^{(0,0,0)},
$$
where $J_x^\top$ and $J_y^\top$ are the Jacobi operators for $\PPP^(1,1,1)$ and
\[
L_{(1,1,1)}^{(0,0,0)}=L_{(0,0,1)}^{(0,0,0)} L_{(0,1,1)}^{(0,0,1)} L_{(1,1,1)}^{(0,1,1)}, \qquad
S_{(0,0,0)}^{(1,1,1)}=S_{(0,1,1)}^{(1,1,1)} S_{(0,0,1)}^{(0,1,1)} S_{(0,0,1)}^{(0,0,0)}.
\]

In~\cref{fig:Helmholtz} we depict the solution for $k = 100$ and plot the timings for construction, factorization, and solution for $k$ between 100 and 300, using polynomials of degree $2k$.
 The build time depends only on the discretization size, so we observe an $\mathcal{O}(k^2)$ cost. 

\begin{figure}[tb]
\begin{center}
\begin{tabular}{cc}
\includegraphics[width=0.45\textwidth]{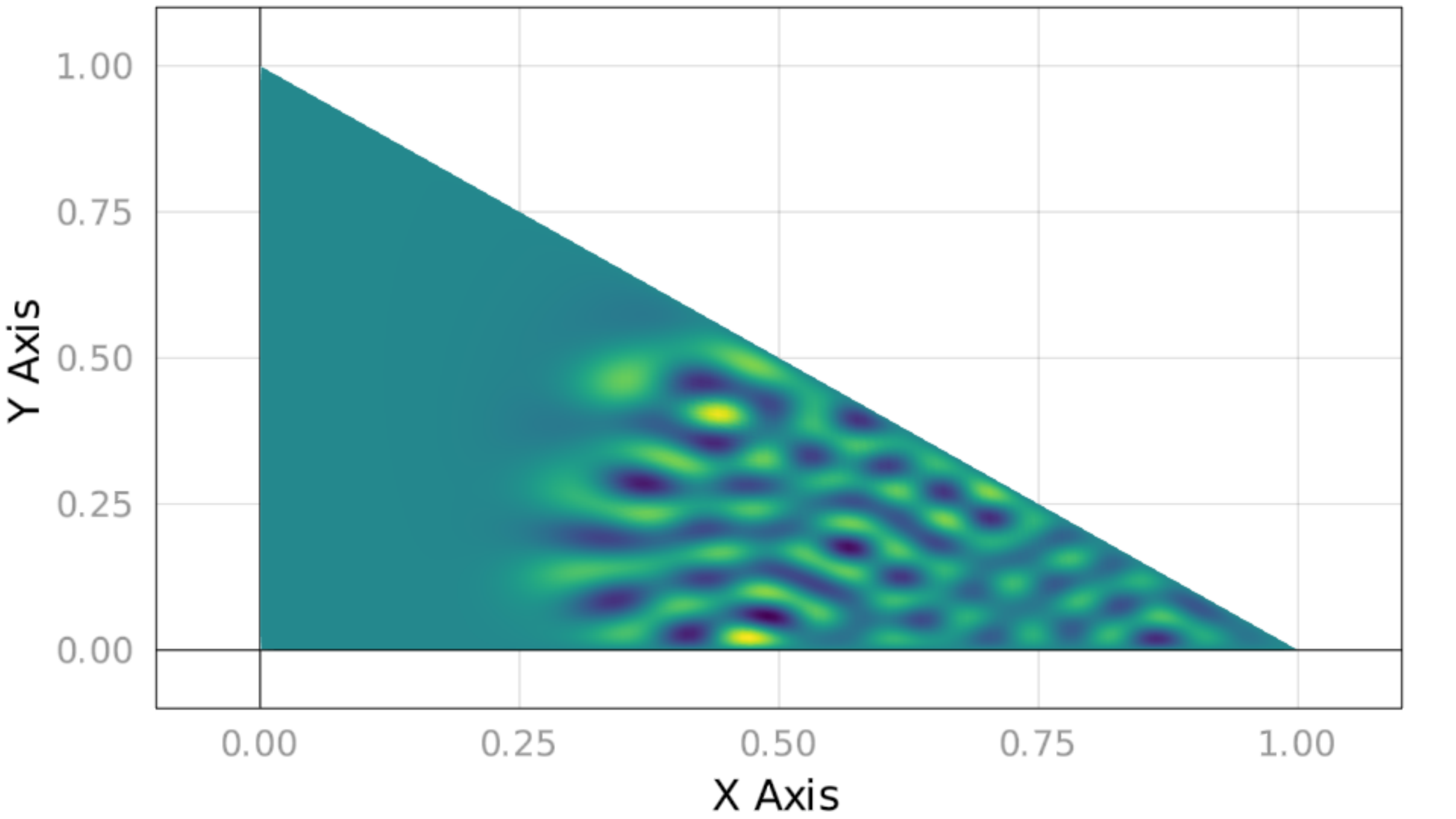}&
\includegraphics[width=0.45\textwidth]{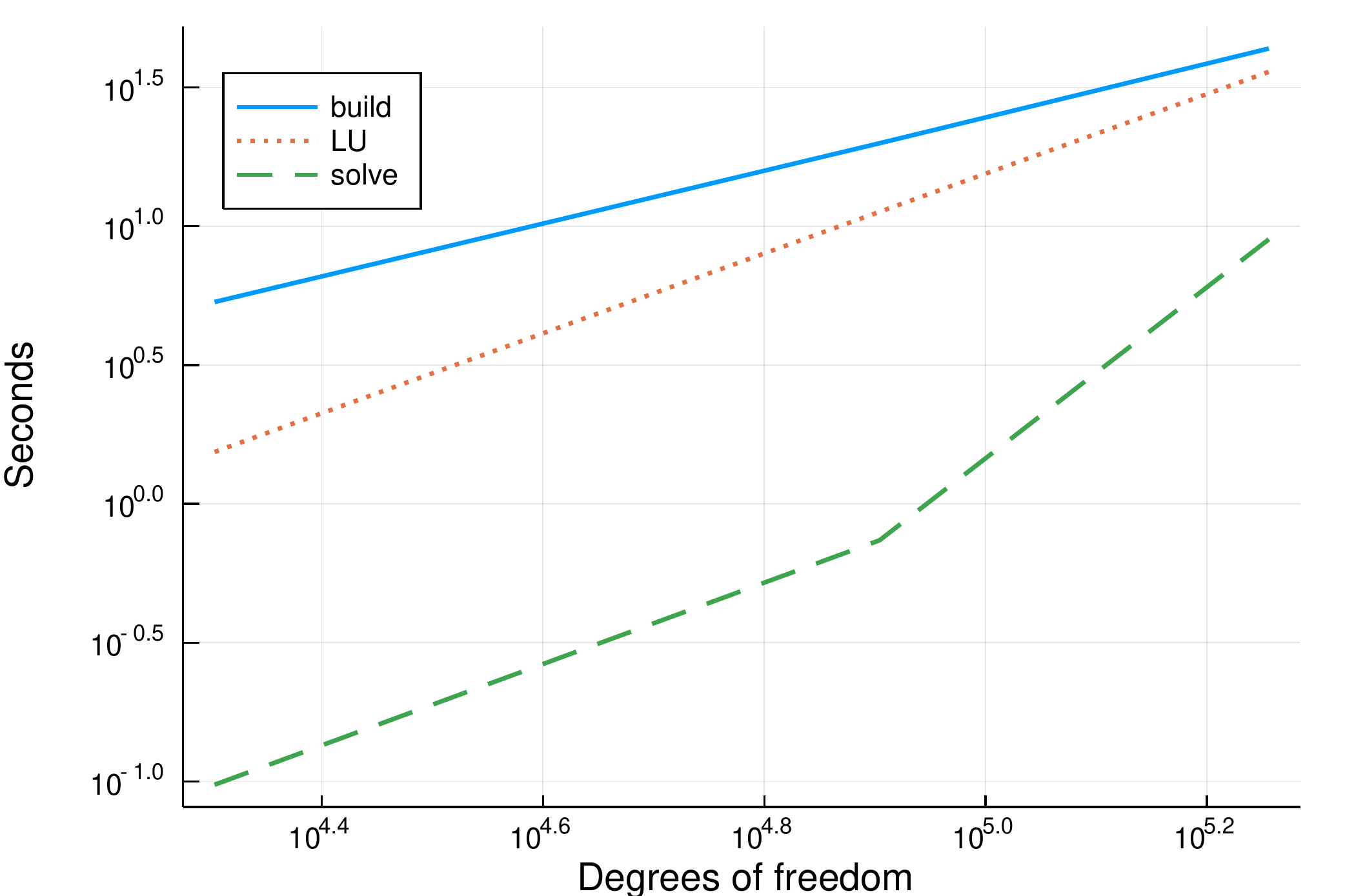}
\end{tabular}
\caption{Left: The computed solution to $(\Delta + k^2 v(x,y)) u = xy\E^x$ with zero Dirichlet conditions and $v(x,y) = 1-(3(x-1)^2 + 5 y^2)$. Right: The execution time to build the discretization, calculate its LU factorization using SuiteSparse, and solve the linear system.}
\label{fig:Helmholtz}
\end{center}
\end{figure}

\subsubsection{Example 3: The biharmonic equation}
The same technique for constructing a sparse representation of the Laplacian $\Delta$ translates to the Biharmonic operator $\Delta^2$, though now we must use a basis that satisfies both zero Dirichlet and Neumann conditions. We can represent the Laplacian as a map from coefficients in an $x^2 y^2 (1-x-y)^2 \PPP^(2,2,2)(x,y)$ expansion to coefficients in an $\PPP^(0,0,0)$ expansion by using weighted differentiation and lowering operators:
$$
\Delta_{W^2} :=  L_{(0,1,0)}^{(0,0,0)} W_{x,(1,1,1)}^{(0,1,0)} L_{(1,2,1)}^{(1,1,1)} W_{x,(2,2,2)}^{(1,2,1)} +  L_{(1,0,0)}^{(0,0,0)} W_{y,(1,1,1)}^{(1,0,0)} L_{(2,1,1)}^{(1,1,1)} W_{y,(2,2,2)}^{(2,1,1)}.
$$
Hence, the biharmonic operator can be sparsely represented as a map from coefficients in an $x^2 y^2 (1-x-y)^2 \PPP^(2,2,2)(x,y)$ expansion to coefficients in an $\PPP^(2,2,2)$ expansion. This is simply given by $\Delta_{(0,0,0)}^{(2,2,2)}  \Delta_{W^2}$.

In~\cref{fig:Biharmonic} we depict a solution to the biharmonic equation and show that the build time grows linearly with respect to the number of degrees of freedom employed to discretize the solution. 
\begin{figure}[tb]
\begin{center}
\begin{tabular}{cc}
\includegraphics[width=0.45\textwidth]{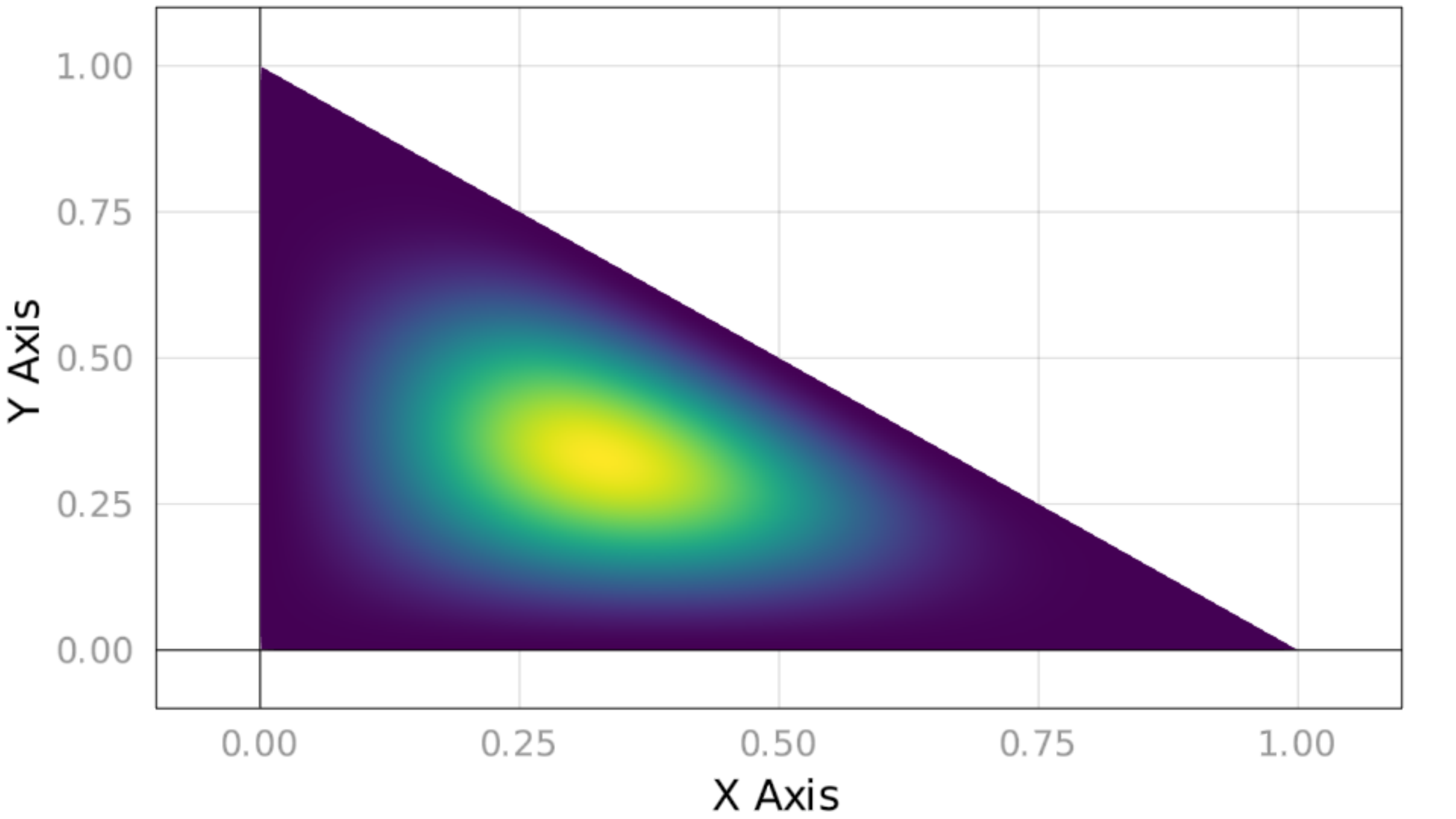}&
\includegraphics[width=0.45\textwidth]{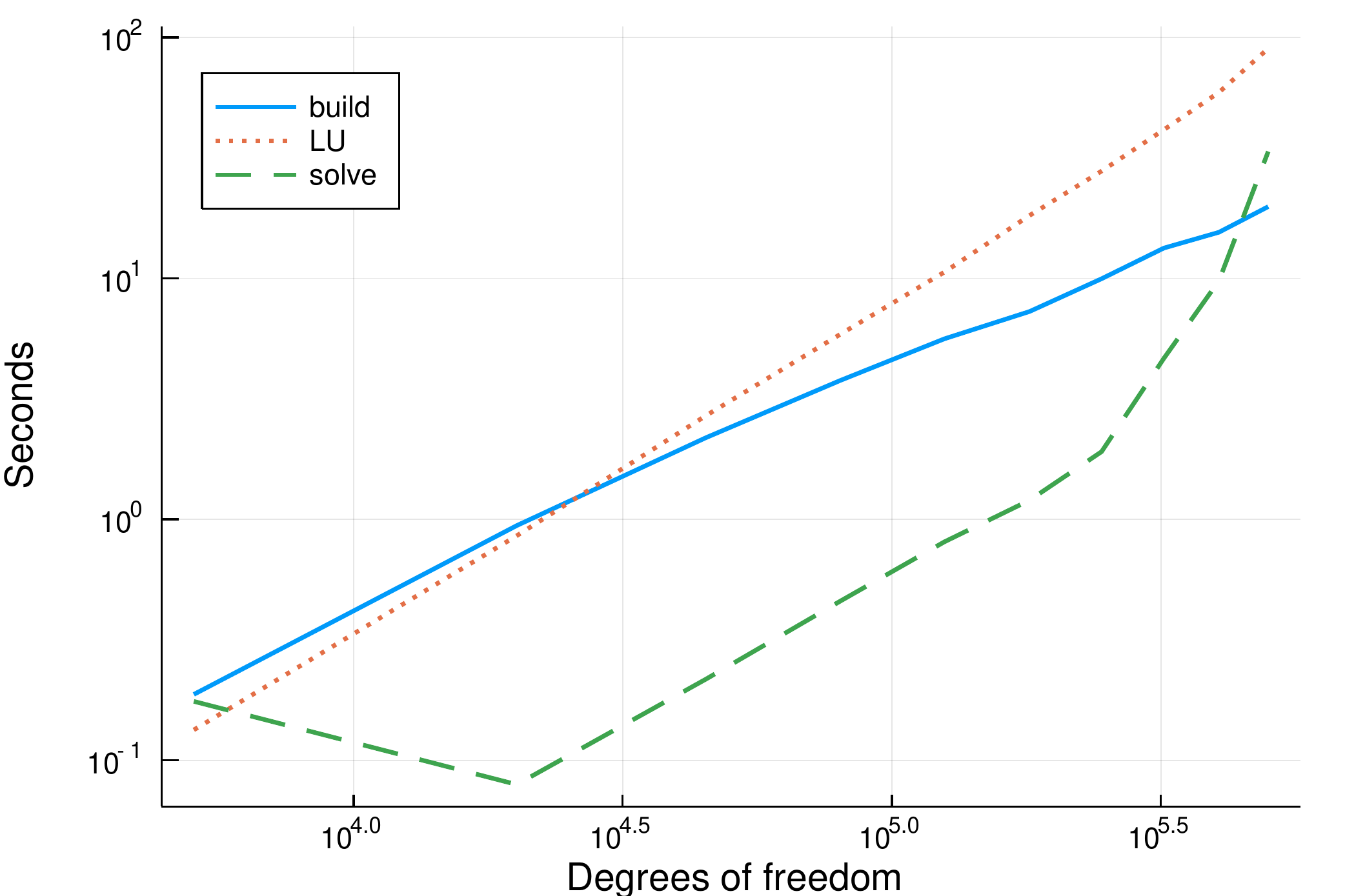}
\end{tabular}
\caption{Left: The  solution to $\Delta^2 u = f$ with zero Dirichlet conditions and $f(x,y) = 1 + \erf(5 (1-10((x-1/2)^2 + (y-1/2)^2))$. Right: The execution time to build the discretization, calculate its LU factorization using SuiteSparse, and solve the linear system.}
\label{fig:Biharmonic}
\end{center}
\end{figure}

\section{Nonzero Dirichlet conditions} \label{Section:dirichlet}
To handle general nonzero Dirichlet boundary conditions, we wish to construct restriction operators that are sparse operators.  To facilitate this, we use a basis where most elements of the basis vanish at the boundary. We take the weighted basis $x^a y^b(1-x-y)^c \PPP^(a,b,c)$, where $a, b,c$ are integers, and augment it with additional polynomials so that the basis can represent all bivariate polynomials.  This is essentially the same procedure as in~\cite{Karniadakis_Sherwin_13}, but we do it in a way that preserves the sparsity of the restriction operators. \Cref{Appendix:DirichletBasis} gives the definition of $Q_{n,k}^{(a,b,c)}(x,y)$, where $a,b,c \in \set{0,1}$, which is the basis we use to construct sparse discretizations. Here, most of $Q_{n,k}^{(1,b,c)}(0,y)$, most of $Q_{n,k}^{(a,1,c)}(x,0)$, and most of $Q_{n,k}^{(a,b,1)}(x,1-x)$ vanish. 

\begin{remark} 
Formally, $Q_{n,k}^{(a,b,c)}(x,y)$ can be thought of as $P_{n,k}^{(-a,-b,-c)}(x,y)$, which is made precise in \cite{Xu_TA} during the construction of the polynomials $J_{n,k}^{(a,b,c)}(x,y)$. However, the construction in \cite{Xu_TA} is normalized in a way that leads to underflow in double precision computing and we find it simpler to define our own basis $Q_{n,k}^{(a,b,c)}(x,y)$ in an ad hoc way.
\end{remark} 

\subsection{Derivative and conversion operators}
Partial derivatives and conversion operators $\vc Q^{(a,b,c)}(x,y)$ are similar to those derived for $\vc P^{(a,b,c)}(x,y)$.  Using the formulas in \cref{Corollary:OneEdgeConversion}, we can construct conversion operators that convert from one-edge bases to $\vc P^{(0,0,0)}$:
\meeq{
f(x,y) = \vc Q^{(1,0,0)}(x,y)^\top \vc f = \vc P^{(0,0,0)}(x,y)^\top \tilde S_{(1,0,0)}^{(0,0,0)} \vc f, \ccr
f(x,y) = \vc Q^{(0,1,0)}(x,y)^\top \vc f = \vc P^{(0,0,0)}(x,y)^\top \tilde S_{(0,1,0)}^{(0,0,0)} \vc f, \ccr
f(x,y) = \vc Q^{(0,0,1)}(x,y)^\top \vc f = \vc P^{(0,0,0)}(x,y)^\top \tilde S_{(0,0,1)}^{(0,0,0)} \vc f. 
}
Note that each operator is block upper bi-diagonal, with diagonal or upper bi-diagonal blocks. Similarly, \cref{Corollary:TwoEdgeConversion} derives sparse conversion operators from two-edge bases to one edge bases, which we denote by $\tilde S_{(1,1,0)}^{(1,0,0)}$, $\tilde S_{(1,1,0)}^{(0,1,0)}$, $\tilde S_{(1,0,1)}^{(1,0,0)}$, etc.  Finally, \cref{Corollary:ThreeEdgeConversion} derives sparse conversion operators from the three-edge basis to any of the two-edge bases, which we denote by $\tilde S_{(1,1,1)}^{(1,1,0)}$, $\tilde S_{(1,1,1)}^{(1,0,1)}$, and $\tilde S_{(1,1,1)}^{(0,1,1)}$. We can clearly compose these operators together to convert from, say, $\vc Q^{(1,1,1)}$ to $\vc P^{(0,0,0)}$. For this purpose, we can define
$$
\tilde S_{(1,1,1)}^{(0,0,0)} := \tilde S_{(1,0,0)}^{(0,0,0)} \tilde S_{(1,1,0)}^{(1,0,0)}\tilde S_{(1,1,1)}^{(1,1,0)}.
$$
There are always several paths through the parameter space to convert one basis into another; however, any path that is chosen results in the same final conversion operator.  

The same principle is true for derivatives, though for our purposes it suffices to restrict our attention to derivatives of the  two-edge bases. \Cref{Corollary:TwoEdgeDerivative} gives us the entries for sparse (block-superdiagonal with at most bidiagonal blocks) operators  that satisfy:
	\meeq{
{\partial f \over \partial x} =  {\partial \over \partial x} \vc Q^{(1,0,1)}(x,y)^\top \vc f  =\PPPt^(0,0,0)(x,y) \tilde D_{x,(1,0,1)}^{(0,0,0)} \FF, \ccr
{\partial f \over \partial y} =  {\partial \over \partial y} \vc Q^{(0,1,1)}(x,y)^\top \vc f  =\PPPt^(0,0,0)(x,y) \tilde D_{y,(0,1,1)}^{(0,0,0)} \FF.
}

General partial derivative operators can be constructed by combining conversion and derivative operators. For example,  we can successfully construct the Laplacian from $\vc Q^{(1,1,1)}$ to $\vc P^{(1,1,1)}$ as
$$
\tilde \Delta := S_{(1,0,1)}^{(1,1,1)} D_{x,(0,0,0)}^{(1,0,1)} \tilde D_{x,(1,0,1)}^{(0,0,0)}\tilde S_{(1,1,1)}^{(1,0,1)} +  S_{(0,1,1)}^{(1,1,1)} D_{y,(0,0,0)}^{(0,1,1)} \tilde D_{y,(0,1,1)}^{(0,0,0)}\tilde S_{(1,1,1)}^{(0,1,1)}.
$$
This is a sparse operator with block-bandwidths and sub-blockbandwidths equal to $(1,4)$.

\subsection{Restriction operators} The definitions of $Q_{n,k}^{(1,0,0)}(x,y)$,  $Q_{n,k}^{(0,1,0)}(x,y)$, and  $Q_{n,k}^{(0,0,1)}(x,y)$ each has a simple restriction formula to one of the three edges of the triangle:
\meeq{
	Q_{n,n}^{(1,0,0)}(0,y) = \tilde P_{n}(y) \qqand Q_{n,k}^{(1,0,0)}(0,y) = 0 \qfor k=0,\ldots,n-1, \ccr
	Q_{n,0}^{(0,1,0)}(x,0) = \tilde P_{n}(x) \qqand Q_{n,k}^{(0,1,0)}(x,0) = 0 \qfor k=1,\ldots,n, \ccr
	Q_{n,0}^{(0,0,1)}(x,1-x) = \tilde P_{n}(x) \qqand Q_{n,k}^{(0,0,1)}(x,1-x) = 0 \qfor k=1,\ldots,n. 
	}
In other words, the restriction operator from expansion in $Q_{n,k}^{(1,0,0)}(x,y)$ to Legendre expansion on the edge from $(0,0)$ to $(0,1)$ is a block-banded operator, where the blocks themselves are very sparse: each block has precisely one nonzero entry. Similarly, the other two bases give restriction operators to the other edges.  We denote these restriction operators as $R_x$, $R_y$, and $R_z$, respectively. They are given by 

\meeq{
	f(0,y) = {\vc Q}^{(1,0,0)}({0,y})^\top \vc f =   {\vc P}(y)^\top R_x \vc f, \ccr
	f(x,0) =  {\vc Q}^{(0,1,0)}({0,y})^\top \vc f =   {\vc P}(x)^\top R_y \vc f,  \ccr
	f(x,1-x) =  {\vc Q}^{(0,0,1)}({x,1-x})^\top \vc f =   {\vc P}(x)^\top R_z \vc f,
}
where  $ {\vc P}(x) := \PPPt^(0,0)(x)$ are the shifted Legendre polynomials. 

For the full Dirichlet operator, we need to restrict to all three edges. Thus we can construct restriction operators from $\vc Q^{(1,1,1)}$ to the boundary, where the boundary basis are piecewise mapped Legendre polynomials.  This restriction operator can be calculated by combining conversion and the one-edge restrictions as follows:
$$
R := \sopmatrix{
			R_x \tilde S_{(1,1,0)}^{(1,0,0)} \tilde S_{(1,1,1)}^{(1,1,0)}\\
			R_y \tilde S_{(1,1,0)}^{(0,1,0)} \tilde S_{(1,1,1)}^{(1,1,0)} \\
			R_z \tilde S_{(1,0,1)}^{(0,0,1)} \tilde S_{(1,1,1)}^{(1,0,1)}
			}.
$$
This operator is also sparse as each component is a product of sparse operators.

\subsection{The $\tau$-method}

An issue we must deal with is boundary data with discontinuities at the corners. Consider, for example,  the Laplace equation with Dirchlet conditions:
$$
\Delta u = 0, \qquad u |_{x = 0} = f,\qquad u |_{x = 0} = g, \qqand u |_{z = 0} = h.
$$
 If the boundary data has discontinuities, that is, $f(0,0) \neq g(0,0)$, $f(0,1) \neq h(0,1)$, or $g(1,0) \neq h(1,0)$, then the solution itself will have an arg-like singularity: e.g., near the origin we have the local behaviour
 $$
 u(x,y) \sim (g(0,0) - f(0,0)){2 \over \pi} \arg( x + \I y) + f(0,0).
 $$
Representing $u(x,y)$ by polynomials is therefore limited as they impose continuity. Other PDEs like the Helmholtz equation have similar before when the boundary data has discontinuities. 

To overcome this issue, we adapt the  Lanczos $\tau$-method, see \cite{OrtizTau} for an overview. The Lanczos $\tau$-method is a device to produce invertible systems for polynomial spectral methods by augmenting the equations with polynomial correction terms, 
that also provide error control by measuring the magnitude of the correction term.  In our context we use it to capture discontinuities  by  augmenting the boundary data with corrections of the form
$$
\Delta u = 0, \qquad u |_{x = 0} = f + \tau_1,\qquad u |_{x = 0} = g + \tau_2, \qqand u |_{z = 0} = h.
$$
That is, we add constants $\tau_1$ and $\tau_2$ to our discretisation:
$$
\sopmatrix{1 & 0 & R_x  \tilde S_{(1,1,0)}^{(1,0,0)} \tilde S_{(1,1,1)}^{(1,1,0)}\\ 0 & 1 &R_y \tilde S_{(1,1,0)}^{(0,1,0)} \tilde S_{(1,1,1)}^{(1,1,0)}  \\ 0 & 0 &R_z \tilde S_{(1,0,1)}^{(0,0,1)} \tilde S_{(1,1,1)}^{(1,0,1)}\\ 0 & 0  & \tilde \Delta}  \Vectt[\tau_1, \tau_2, \vc u] = \Vectt[\vc f ,\vc g , \vc h, 0]
$$

Now in our examples below we actually have mathematically continuous boundary data, however, round-off errors mean our boundary data is slightly discontinuous. The $\tau$ correction terms give a way of capturing this discontinuity without destroying the regularity of $u$. When the solution is resolved the $\tau$ terms are therefore negligible, and we can use the approximation of $u$ on its own.

Note that it is possible to add additional $\tau$ correction terms to make the system invertible, but this is a more technical task and hence we prefer to use a QR decomposition to solve the resulting rectangular linear system in a least squares sense. This does incur a substantial penalty, as SuiteSparse's QR decomposition is significantly slower than its LU decomposition.

\subsubsection{Example 4: Laplace's equation} Consider Laplace's equation with prescribed Dirichlet data:
$$
u_{xx} + u_{yy} = 0, \qquad u(0,y) = f(y), \quad  u(x,0) = g(x), \quad u(x,1-x) = h(x).
$$
Expanding $f(x)$, $g(y)$, and $h(x)$ in Legendre series leads to a system of equations satisfied by $u$. That is, 
$$
\sopmatrix{R \\ \tilde \Delta} \vc u = \Vectt[ \vc f, \vc g, \vc h, 0],
$$
where $\vc u$ are the coefficients of $u(x,y)$ in the basis $\vc Q^{(1,1,1)}$, $\vc f$ are the Legendre coefficients of $f(y)$, $\vc g$ are the Legendre coefficients of $g(x)$, and $\vc h$ are the Legendre coefficients of $h(x)$. We augment this system with $\tau$ corrections, which are ultimately ignored in the approximation of the solution. 

In \cref{fig:LaplaceError} we plot the calculated coefficients for $N = 999$ for three choices of boundary data: $\E^x \cos y$, $x^2$, and $x^3 (1-x)^3 (1-y)^3$. The first two examples exhibit algebraic decay, with the rate of decay dictated by the number of derivatives matching at the corners.  The last example has a smooth solution ($\E^x \cos y$ is harmonic) and we see that the algorithm achieves super-algebraic convergence, and is stable for large $N$.  We also note that evaluating the approximation is exact to within an accuracy of $3 \times 10^{-16}$ compared to the exact solution at the arbitrary point $(x,y) = (0.1,0.2)$. 

\begin{figure}[tb]
\begin{center}
\begin{tabular}{cc}
\includegraphics[width=0.45\textwidth]{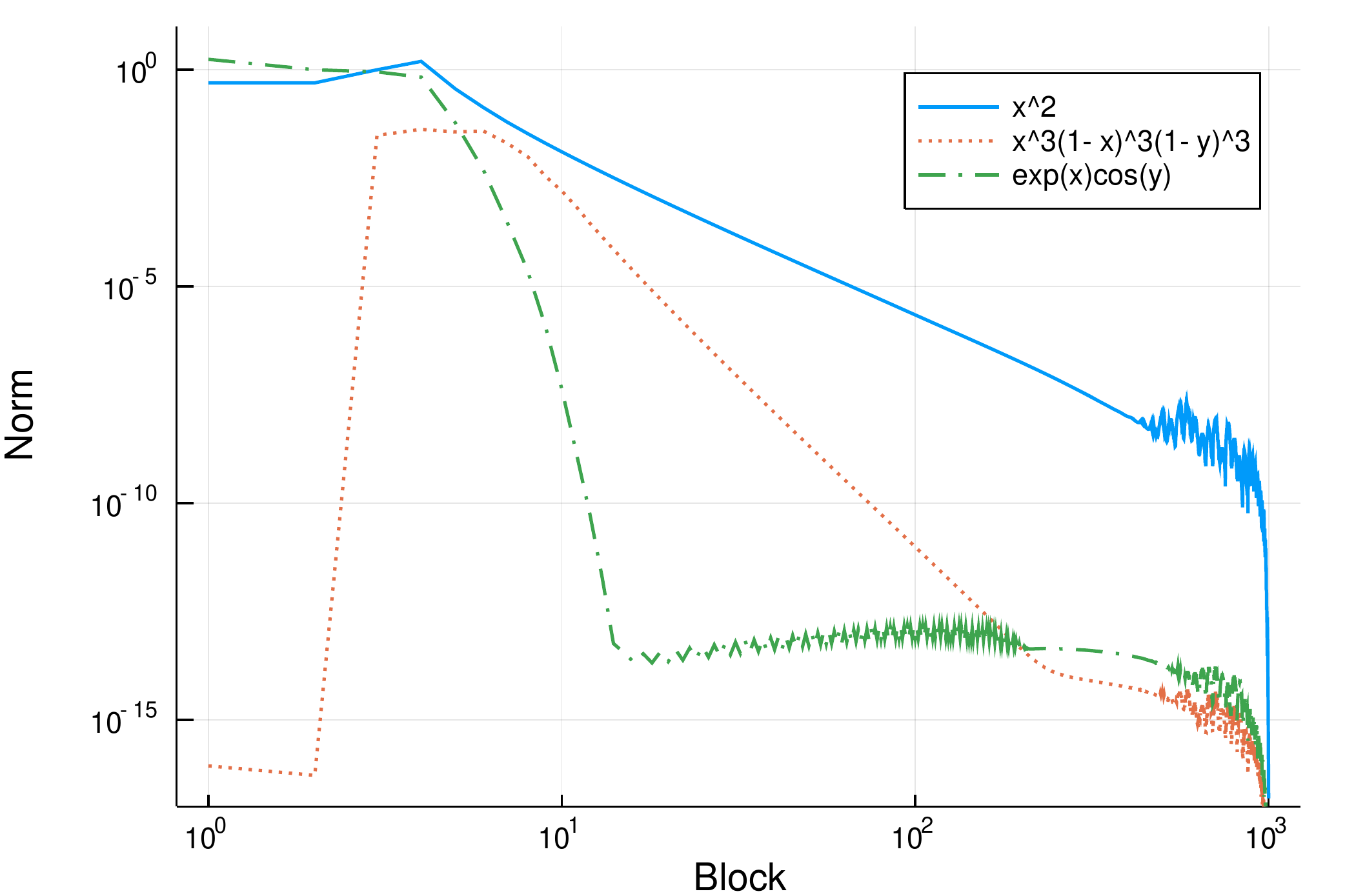}
\end{tabular}
\caption{The norm of the blocks of the calculated coefficients to the solution of $\Delta u = 0$ with specified Dirichlet boundary conditions with $N = 1000$.   The first two examples show algebraic convergence, with faster convergence when there is more continuity at the corners. The third example shows spectral convergence when the solution is smooth. 
}
\label{fig:LaplaceError}
\end{center}
\end{figure}

\subsubsection{Example 5: Transport equation}

\begin{figure}[htbp]
\begin{center}
\begin{tabular}{ccc}
\includegraphics[width=0.3\textwidth]{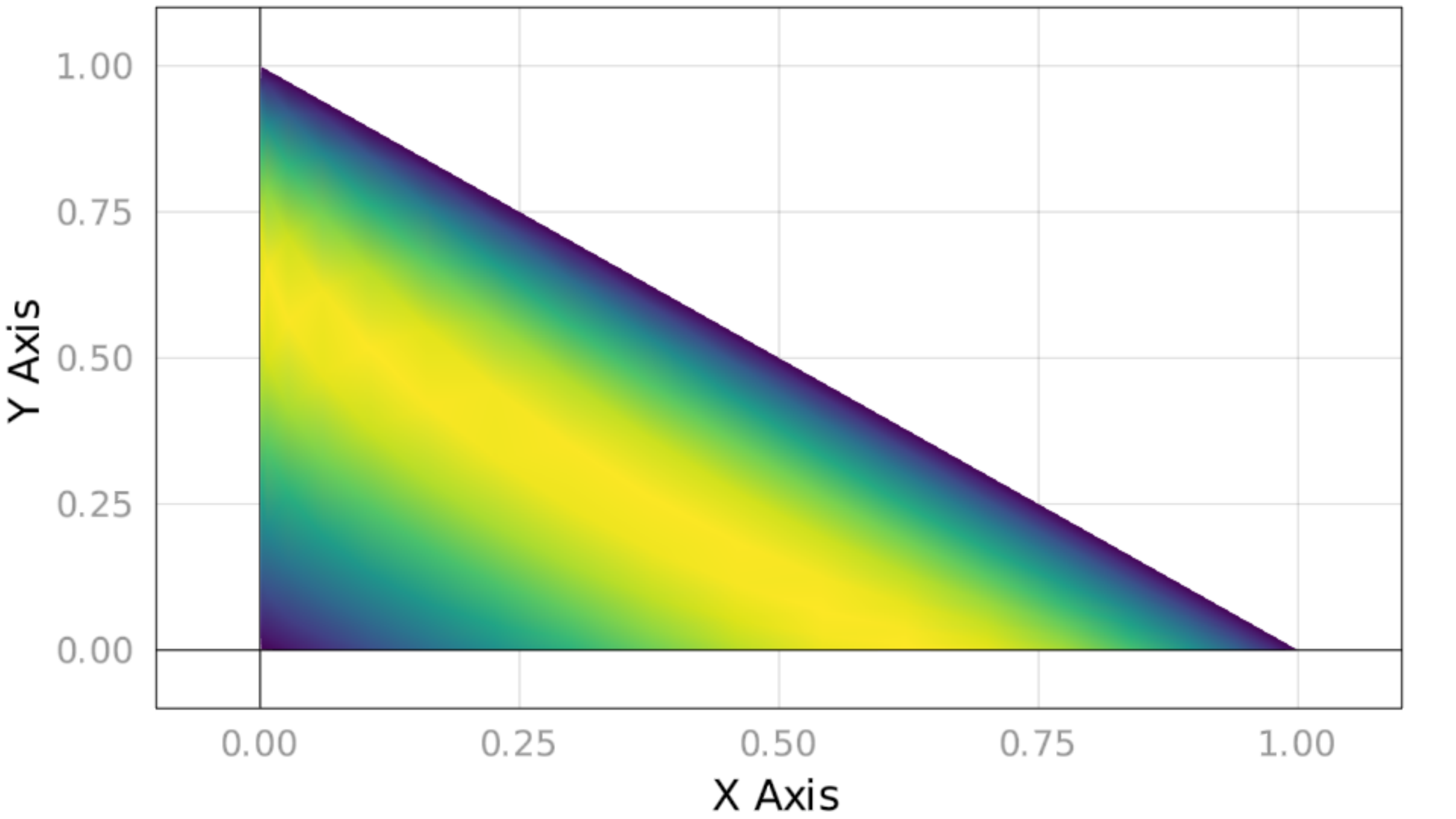}&
\includegraphics[width=0.3\textwidth]{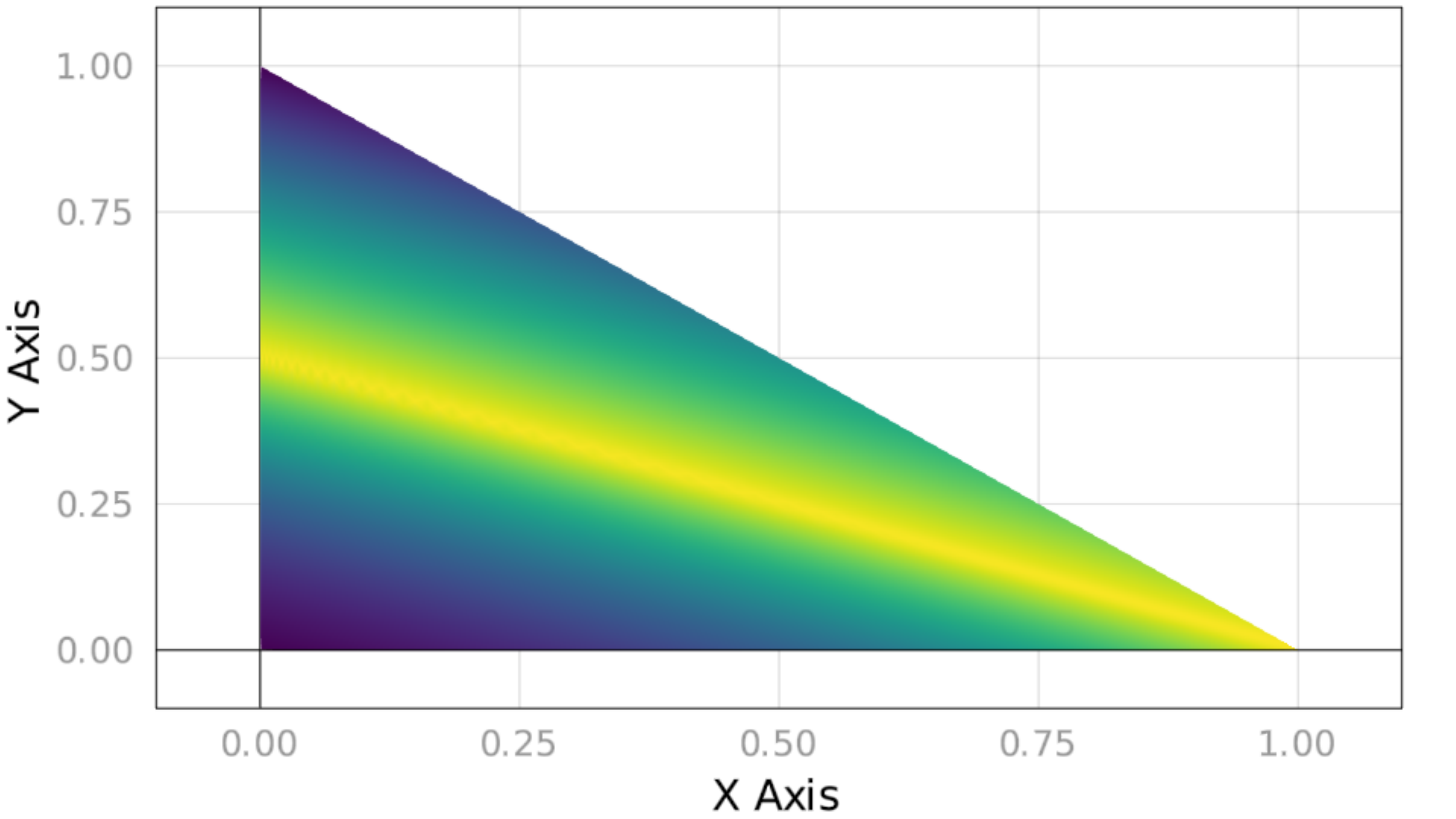}&
\includegraphics[width=0.3\textwidth]{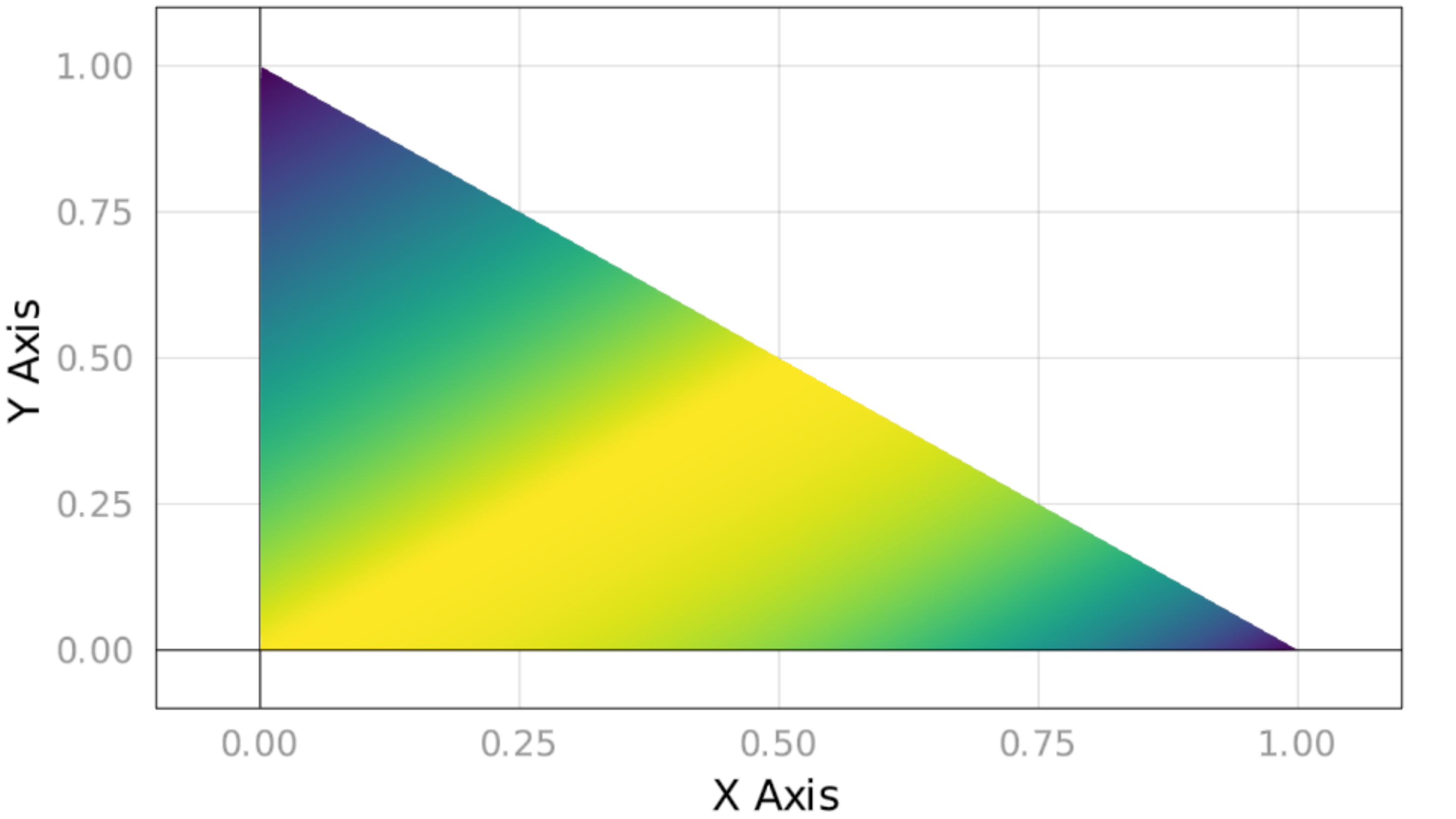}
\end{tabular}
\caption{Left: The  solution to $u_y =  u_x$  using the basis $\vc Q^{(0,1,0)}$ with boundary condition $u(x,0) =x(1-x) \E^x$ imposed on the bottom. Middle: The  solution to $u_y =  2u_x$  using the basis $\vc Q^{(0,1,1)}$ with boundary condition $u(x,0) =x \E^{x-1}$ imposed on the bottom and $u(x,1-x) = x$ on the hypotenuse. Right: The  solution to $u_y =  -u_x$  using the basis $\vc Q^{(1,1,0)}$ with boundary condition $u(x,0) = (1-x) \E^x$ imposed on the bottom and $u(0,y) = 1-y$ on the left.}
\label{fig:Transport}
\end{center}
\end{figure} 

Nothing in this framework depends on the PDE being elliptic. Here, we consider the transport equation given by
$$u_y = c u_x.$$
Information travels at a rate and direction dictated by $c$, and depending on its value we need either one or two edges to uniquely determine the solution. If $0 \leq c \leq 1$ the solution is uniquely determined from the boundary on the bottom, and hence we use the basis $Q^{(0,1,0)}$. If $c > 1$ then information is coming in from the right, so we use the basis $Q^{(0,1,1)}$ on the bottom and hypotenuse edges. If $c < 0$ then information comes in from the left and we use the basis $Q^{(1,1,0)}$ on the bottom and left edges.  In~\cref{fig:Transport} we depict the three solutions. 

\section{Systems of PDEs} \label{Section:systems}
Systems of PDEs can be handled in a straightforward way by concatenating their blocks. As an example, we can  solve the Poisson equation with Neumann conditions by re-expressing the PDE as a first-order system: writing $v = u_x$ expressed in the basis $\vc Q^{(1,0,1)}$,  and $w = u_y$ expressed in the basis $\vc Q^{(0,1,1)}$,  the system becomes 
$$
\begin{pmatrix}
   0 & -R_x \tilde S_{(1,0,1)}^{(1,0,0)} & 0  \\
   0 & 0  & -R_y\tilde  S_{(0,1,1)}^{(0,1,0)} \\
   0 & R_z  \tilde S_{(1,0,1)}^{(0,0,1)} & R_z \tilde S_{(0,1,1)}^{(0,0,1)} \\
    D_{x,(0,0,0)}^{(1,0,1)} & -S_{(0,0,0)}^{(1,0,1)} \tilde S_{(1,0,1)}^{(0,0,0)} & 0 \\
   D_{y,(0,0,1)}^{(0,1,1)} & 0 & -S_{(0,0,0)}^{(0,1,1)} \tilde S_{(0,1,1)}^{(0,0,0)}  \\
   0 & \tilde D_{x,(1,0,1)}^{(0,0,0)} & \tilde D_{y,(0,1,1)}^{(0,0,0)}
  \end{pmatrix}\! \Vectt[u,v,w] = \Vectt[0,0,0,0,0, \vc f]
$$

\subsubsection{Example 6: Helmholtz equation in a polygon}
\begin{figure}[htbp]
\begin{center}
\begin{tabular}{cc}
\includegraphics[width=0.45\textwidth]{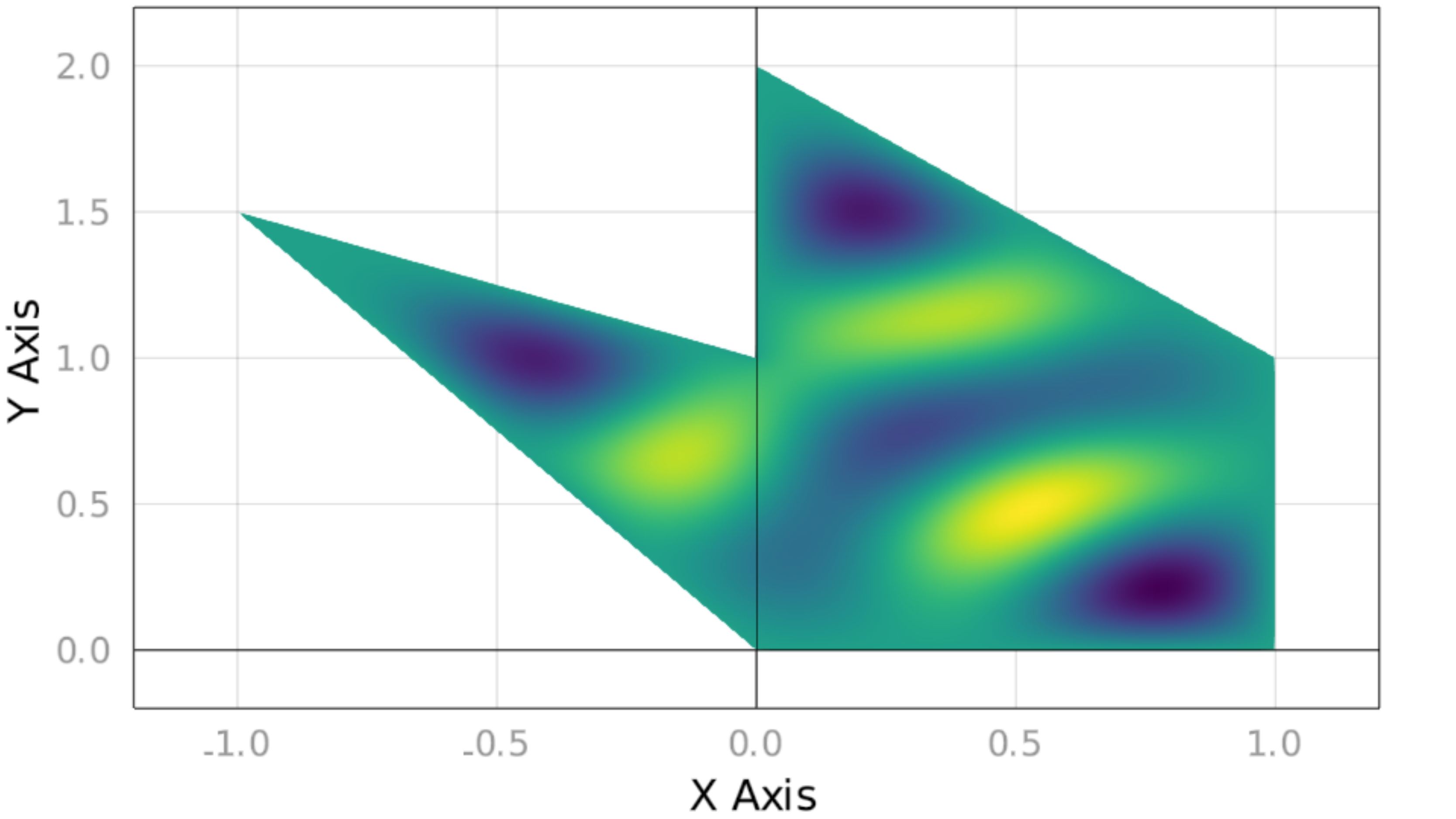}&
\includegraphics[width=0.45\textwidth]{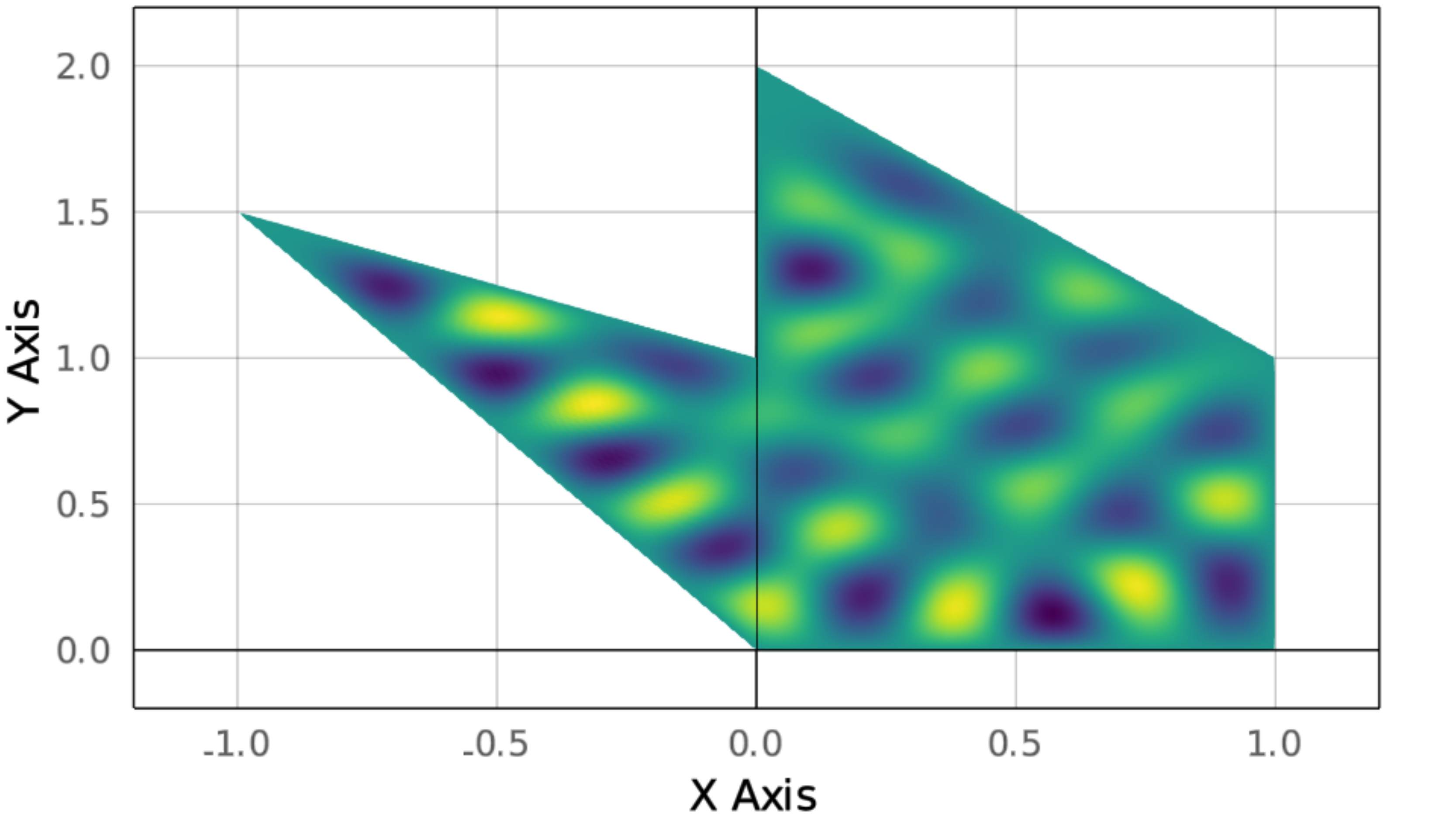}
\end{tabular}
\caption{The  solution to $(\Delta^2  + k^2 )  u = 0$ with Dirichlet boundary conditions fixed to one, for $ k= 10 $ (left) and $ 20$ (right).
}
\label{fig:HelmholtzPolygon}
\end{center}
\end{figure} 

Note that being able to handle systems of PDEs in this manner also allows us to solve on polygonal domains that are partitioned into triangular elements. For example, consider the Helmholtz equation 
$$
u_{xx} + u_{yy} + k^2 u = 0
$$
on the polygonal domain with the vertices $(0,0)$, $(1,0)$, $(1,1)$, $(0,2)$, $(0,1)$, and $(-1,1.5)$. We can decompose this domain into 4 triangles and represent the solution as well as its first derivatives in orthogonal polynomial expansions. This leads to a system of $4 \times 3 = 12$ PDEs. We then impose continuity of the value and the normal derivative across the interfaces of each element, exploiting the fact that the restriction operator maps to the same basis of Legendre polynomials. (The orientation may be different, but reversing orientation of Legendre expansions corresponds to multiplying by a diagonal matrix that swaps the signs of every other coefficients.)   We show the success of this approach in~\cref{fig:HelmholtzPolygon} for $k=10$ (left) and $k=20$ (right).

The discretization of the PDE system is sparse, and the complexity of building the matrices is an optimal  $\mathcal{O}(N^2)$ using degree $N$ polynomials within each element. 

\section{Conclusions}\label{Section:future}
We have shown that bivariate orthogonal polynomials can lead to sparse discretizations of general linear PDEs on triangles with Dirichlet and Neumann boundary conditions. Instead of quadrature, we use sparse recurrence relationships combined with specialized linear algebra routines, allowing optimal complexity for building the linear systems.  Multiple triangles can be patched together to solve PDEs on polygonal domains.  

Another extension is to tetrahedra in 3D and higher. We expect this to be straightforward because the definitions of orthogonal polynomials on higher dimensional simplices is very similar to the 2D  case. In 3D, we can use the following polynomials: 
$$
P_{n,k,j}^{(a,b,c,d)}(x,y,z) := P_{n-j,k}^{(a,b,2j+c+d+1)}(x,y) (1-x-y)^j P_j^{(d,c)}\pr({z \over 1-x-y}),
$$
which are orthogonal with respect to $x^a y^b z^c (1-x-y-z)^d$ on the unit 3D simplex.  The most time-consuming part of such an extension is deriving the recurrences relationships. Note that in 3D and higher the sparsity of our construction is useful even for small discretisation sizes, as a degree $N$ dense discretisation (e.g. arising from collocation) would require calculating $\mathcal{O}(N^6)$ entries, where the proposed construction would require an optimal $\mathcal{O}(N^3)$ operations. 

We used direct solvers via SuiteSparse to solve the resulting discretizations, which is fairly efficient with even millions of unknowns. However, to push the methodology further we will need robust iterative methods and the development of preconditioners. It is not yet clear how to design preconditioners in this setting. 

\section*{Acknowledgments} 
This work began when the second author visited the first author at The University of Sydney. We are grateful for the travel support provided by The University of Sydney. We also thank Andrew Horning and Nicolas Boulle for carefully reading the draft and improving the text. 

%
%
%
%
%
%
%

\appendix

%
%

\section{Recurrence relationships for Jacobi polynomials on the triangle} \label{Appendix:PRecurrences}
Here, we outline the recurrence relationships for $P_{n,k}^{(a,b,c)}(x,y)$ that we employ, which were previously derived in~\cite{Olver_18_01,Xu_TA}. We define $z := 1-x-y$ and ${\partial \over \partial z} := {\partial \over \partial y} - {\partial \over \partial x}$.

\begin{corollary}\cite[Corollary 1]{Olver_18_01}\label{Corollary:Differentiation}
The following recurrence relations for the partial derivatives hold:
\begin{align*}
(2k+b+c+1) \ddx \PPabc &= (n+k+a+b+c+2)(k+b+c+1){P}_{n-1,k}^{(a+1,b,c+1)} \notag\\
&\qquad  \qquad + (k+b)(n+k+b+c+1){P}_{n-1,k-1}^{(a+1,b,c+1)}, \\
\ddy \PPabc  &= (k + b + c + 1) {P}_{n-1,k-1}^{(a,b+1,c+1)},  \label{ddy}\\
(2k+b+c+1)\ddz\PPabc  & = -(n+k+a+b+c+2) (k+b+c+1) {P}_{n-1,k}^{(a+1,b+1,c)} \notag \\
   & \qquad \qquad + (k+c)(n+k+b+c+1) {P}_{n-1,k-1}^{(a+1,b+1,c)}.   
\end{align*}
\end{corollary}

\begin{corollary}\cite[Corollary 2]{Olver_18_01}\label{Corollary:WeightedDifferentiation}
The following recurrence relations for the weighted partial derivatives hold:
$$
\begin{aligned}
	-(2k+b+c+1)\ddx\!\left(x^a y^b z^c \PPabc\right) &= x^{a-1}y^b z^{c-1}\Big( (k+c)(n-k+1){P}_{n+1,k}^{(a-1,b,c-1)} \\
	&\qquad\qquad + (k+1)(n-k+a){P}_{n+1,k+1}^{(a-1,b,c-1)}\Big), \\
	\ddy\!\left(x^a y^b z^c \PPabc\right) &= -(k+1)x^ay^{b-1} z^{c-1} {P}_{n+1,k+1}^{(a,b-1,c-1)}, \\
	(2k+b+c+1) \ddz\! \left(x^a y^b z^c \PPabc\right) &=  x^{a-1}y^{b-1} z^c \Big((k+b)(n-k+1){P}_{n+1,k}^{(a-1,b-1,c)} \\
	&\qquad\qquad - (k+1)(n-k+a){P}_{n+1,k+1}^{(a-1,b-1,c)}\Big).
\end{aligned} 
$$
\end{corollary}

\begin{corollary}\cite[Corollary 3]{Olver_18_01}\label{Corollary:Conversion}
The following recurrence relations for conversions hold:
\begin{align*}
(2n+a+b+c+2)&\PPabc = \\
&(n+k+a+b+c+2){P}_{n,k}^{(a+1,b,c)} \cr
+& (n+k+b+c+1){P}_{n-1,k}^{(a+1,b,c)},\\
(2n+a+b+c+2)&(2k+b+c+1)\PPabc = \\
&(n+k+a+b+c+2)(k+b+c+1){P}_{n,k}^{(a,b+1,c)} \notag\\
- &(n-k+a)(k+b+c+1){P}_{n-1,k}^{(a,b+1,c)}\notag\\
+&(k+c)(n+k+b+c+1){P}_{n-1,k-1}^{(a,b+1,c)}\notag\\
-&(k+c)(n-k+1) {P}_{n,k-1}^{(a,b+1,c)},\\
(2n+a+b+c+2)&(2k+b+c+1)\PPabc  = \\
& (n+k+a+b+c+2)(k+b+c+1){P}_{n,k}^{(a,b,c+1)} \notag\\
-& (n-k+a)(k+b+c+1){P}_{n-1,k}^{(a,b,c+1)}\notag\\
-&(k+b)(n+k+b+c+1){P}_{n-1,k-1}^{(a,b,c+1)}\notag\\
+&(k+b)(n-k+1){P}_{n,k-1}^{(a,b,c+1)}.
\end{align*}
\label{cor:conversion}
\end{corollary}

\begin{corollary}\cite[Corollary 4]{Olver_18_01}\label{Corollary:Lowering}
The following recurrence relations for lowering operators hold:
\begin{align*}
	(2n+a+b+c+2)  x \PPabc = (n-k+a) {P}_{n,k}^{(a-1,b,c)} +(n-k+1) {P}_{n+1,k}^{(a-1,b,c)}, \label{Mx}
	\end{align*}
\begin{align*}
 (2k+b+c+1)&(2n+a+b+c+2) y \PPabc = \\
 &(k+b)(n+k+b+c+1){P}_{n,k}^{(a,b-1,c)}\cr
 -&(k+1)(n-k+a){P}_{n,k+1}^{(a,b-1,c)}\cr
-&(k+b)(n-k+1){P}_{n+1,k}^{(a,b-1,c)}\cr
+& (k+1)(n+k+a+b+c+2){P}_{n+1,k+1}^{(a,b-1,c)}, 
\label{My}
\end{align*}
\begin{align*}
	(2 k + b + c+1)& (2n + a + b + c + 2)z \PPabc  = \\
	&(k + c)(n + k + b + c + 1){P}_{n,k}^{(a,b,c-1)} \cr
	+&  (k + 1)(n - k + a) {P}_{n,k+1}^{(a,b,c-1)}\cr
	-&  (k + c) (n - k + 1){P}_{n+1,k}^{(a,b,c-1)}  \cr
	-& (k + 1)(n + k + a + b + c + 2){P}_{n+1,k+1}^{(a,b,c-1)}. 
	 \label{Mxy}
\end{align*}
\label{cor:mult}
\end{corollary}

\section{Dirichlet basis definitions}\label{Appendix:DirichletBasis} 
Here, we define a basis, denoted by $Q_{n,k}^{(a,b,c)}(x,y)$, that we employ to impose general Dirichlet and Neumann boundary conditions. We construct $Q_{n,k}^{(a,b,c)}(x,y)$ by augmenting the weighted basis 
$$
x^a y^b z^c P_{n,k}^{(a,b,c)}(x,y)
$$
so that it spans all the polynomials with $a,b,c\in \set{0,1}$. Depending on the choice of $a$, $b$, and $c$, we obtain sparse restriction operators to one, two, or three edges of the triangle. We refer to this basis as the {\it Dirichlet basis} for its usefulness in solving PDEs with Dirichlet and Neumann boundary conditions. 

\subsection{One-edge Dirichlet basis}

\begin{definition}
The following polynomials vanish at $x = 0$ apart from when $k=n$:
\begin{align*}
		Q_{0,0}^{(1,0,0)}(x,y) &:= 1, \cr
		Q_{n,k}^{(1,0,0)}(x,y) &:= x P_{n-1,k}^{(1,0,0)}(x,y)   \qquad \hbox{ for $k = 0,\ldots,n-1$}, \cr
		Q_{n,n}^{(1,0,0)}(x,y) &:= P_{n,n}(x,y).
\end{align*}
The following polynomials vanish at $y = 0$ apart from when $k=0$:
\begin{align*}
		Q_{n,0}^{(0,1,0)}(x,y) &:= \tilde P_n^{(0,0)}(x),	 \cr
		Q_{n,k}^{(0,1,0)}(x,y) &:= y P_{n-1,k-1}^{(0,1,0)}(x,y)   \qquad \hbox{ for $k = 1,\ldots,n$} .
\end{align*}
The following polynomials vanish at $z = 0$ (i.e., $y = 1-x$) apart from when $k=0$:
\begin{align*}
		Q_{n,0}^{(0,0,1)}(x,y) &:= \tilde P_n^{(0,0)}(x),	 \cr
		Q_{n,k}^{(0,0,1)}(x,y) &:= z P_{n-1,k-1}^{(0,0,1)}(x,y)   \qquad \hbox{ for $k = 1,\ldots,n$}.
\end{align*}	
			
\end{definition}					
The ordering is chosen so that the conversion operators derived below are upper triangular.   Each basis has a simple restriction formula to the corresponding edge.

\begin{proposition}
Restriction operator to $x = 0$:
\begin{align*}
			Q_{n,n}^{(1,0,0)}(0,y) &:= \tilde P_{n}^{(0,0)}(y), \cr
		Q_{n,k}^{(1,0,0)}(0,y) &:= 0  \qquad \hbox{for $k \neq n$}.
\end{align*}
Restriction operator to $y = 0$:
\begin{align*}
			Q_{n,0}^{(0,1,0)}(x,0) &:= \tilde P_{n}^{(0,0)}(x), \cr
			Q_{n,k}^{(0,1,0)}(x,0) &:= 0  \qquad \hbox{for $k \neq 0$}.
\end{align*}
Restriction operator to $z = 0$:	
\begin{align*}
			Q_{n,0}^{(0,0,1)}(x,1-x) &:= \tilde P_{n}^{(0,0)}(x), \cr
			Q_{n,k}^{(0,0,1)}(x,1-x) &:= 0  \qquad \hbox{for $k \neq 0$}.
\end{align*}

\end{proposition}


\subsection{Two-edge Dirichlet basis}

To handle two edges, consider first $x=0$ and $y=0$.  As before, we wish to construct a basis that adds in the missing polynomials to $xyP_{n,k}^{(1,1,0)}(x,y)$ in a way that the restriction operators have the necessary structure.  To do this, we select polynomials so that we can construct the conversion operator to expansions in the basis $\vc Q^{(1,0,0)}$ and use the restriction operators we already have (see~\cref{Corollary:OneEdgeConversion}). 
\begin{definition} The following polynomials vanish at $x = 0$ and $y = 0$ apart from when $k = 0, n$:
\begin{align*}
		Q_{0,0}^{(1,1,0)}(x,y) &:= 1,	 \cr	
			Q_{n,0}^{(1,1,0)}(x,y) &:= x \tilde P_{n-1}^{(0,1)}(x),	 \cr
		Q_{n,k}^{(1,1,0)}(x,y) &:= x y P_{n-2,k-1}^{(1,1,0)}(x,y)   \qquad \hbox{ for $k = 1,\ldots,n-1$}, \cr
		Q_{n,n}^{(1,1,0)}(x,y) &:= y P_{n-1,n-1}^{(0,1,0)}(x,y).
\end{align*}			
 The following polynomials vanish at $x = 0$ and $z = 0$ apart from when $k = 0, n$:
\begin{align*}
		Q_{0,0}^{(1,0,1)}(x,y) &:= 1,	 \cr	
			Q_{n,0}^{(1,0,1)}(x,y) &:= x \tilde  P_{n-1}^{(0,1)}(x),		 \cr
		Q_{n,k}^{(1,0,1)}(x,y) &:= x z P_{n-2,k-1}^{(1,0,1)}(x,y)   \qquad \hbox{ for $k = 1,\ldots,n-1$}, \cr
		Q_{n,n}^{(1,0,1)}(x,y) &:= z P_{n-1,n-1}^{(0,0,1)}(x,y).
\end{align*}				
 The following polynomials vanish at $y = 0$ and $z = 0$ apart from when $k = 0, 1$:
\begin{align*}
		Q_{0,0}^{(0,1,1)}(x,y) &:= 1,	 \cr	
			Q_{n,0}^{(0,1,1)}(x,y) &:= (1-x)P_{n-1,0}(x,y) = (1-x)\tilde P_{n-1}^{(1,0)}(x), \cr
			Q_{n,1}^{(0,1,1)}(x,y) &:= (1-x-2y) P_{n-1,0}(x,y) =(1-x-2y) \tilde P_{n-1}^{(1,0)}(x),	 \cr
		Q_{n,k}^{(0,1,1)}(x,y) &:=  y z P_{n-2,k-2}^{(0,1,1)}(x,y)   \qquad \hbox{ for $k = 2,\ldots,n$}.
\end{align*}							
\end{definition}

\subsection{Three-edge Dirichlet basis}
We finally get to three edges.  Again, we want to choose the extra polynomials so that we can easily convert to any two-edge cases.  The following does the trick:
\begin{definition}
The following polynomials vanish at $x = 0$, $y = 0$, and $z = 0$ apart from when $k = 0,1$, and $ n$:
\begin{align*}
		Q_{0,0}^{(1,1,1)}(x,y) &:= 1,	 \cr	
			Q_{1,0}^{(1,1,1)}(x,y) &:= 1-2x,  \cr		
			Q_{1,1}^{(1,1,1)}(x,y) &:= 1-x-2y,  \cr					
			Q_{n,0}^{(1,1,1)}(x,y) &:= x (1-x) P_{n-2,0}^{(1,0,0)}(x,y) = x(1-x) P_{n-2}^{(1,1)}(x), \cr
			Q_{n,1}^{(1,1,1)}(x,y) &:=x (1-x-2y)  P_{n-2,0}^{(1,0,0)}(x,y)	= x (1-x-2y) P_{n-2}^{(1,1)}(x),  \cr
		Q_{n,k}^{(1,1,1)}(x,y) &:=  x y z P_{n-3,k-2}^{(1,1,1)}(x,y)   \qquad \hbox{ for $k = 2,\ldots,n-1$}, \cr
			Q_{n,n}^{(1,1,1)}(x,y) &:= y z P_{n-2,n-2}^{(0,1,1)}(x,y). 
\end{align*}
\end{definition}

\section{Dirichlet basis recurrence relationships}\label{Appendix:DirichletConversion}
The following allows us to construct sparse conversion operators from the one-edge Dirichlet basis to the standard Jacobi polynomials on the triangle:
\begin{corollary}\label{Corollary:OneEdgeConversion}
The following recurrence relationships hold: 
\meeq{
	Q_{0,0}^{(1,0,0)}(x,y) = P_{0,0}(x,y), \ccr
	(2n+1)Q_{n,k}^{(1,0,0)}(x,y) = (n-k)\br[P_{n,k}(x,y) + P_{n-1,k}(x,y) ],\ccr
	Q_{n,n}^{(1,0,0)}(x,y) = P_{n,n}(x,y), \ccr
	(2n+1)Q_{n,0}^{(0,1,0)}(x,y) = (n+1)P_{n,0}(x,y)-nP_{n-1,0}(x,y), \ccr	
	(2n+1)Q_{n,k}^{(0,1,0)}(x,y) = (n+k+1) P_{n,k}(x,y) -(n-k+1) P_{n,k-1}(x,y)\\
	& -(n-k)P_{n-1,k}(x,y)+(n+k) P_{n-1,k-1}(x,y), \ccr
	(2n+1)Q_{n,0}^{(0,0,1)}(x,y) = (n+1)P_{n,0}(x,y)-nP_{n-1,0}(x,y), \ccr	
	(2n+1)Q_{n,k}^{(0,0,1)}(x,y) = -(n+k+1) P_{n,k}(x,y)-(n-k+1) P_{n,k-1}(x,y)  \\
	&+(n-k)P_{n-1,k}(x,y)+(n+k) P_{n-1,k-1}(x,y).
}
\end{corollary}
\begin{proof}
These are either immediate from definitions or are obtained by rearranging recurrence relationships found in \cref{Corollary:Lowering}. 
\end{proof}

The two-edge Dirichlet basis satisfy several sparse recurrence relationships. 
\begin{corollary}\label{Corollary:TwoEdgeConversion}
The following recurrence relationships hold: 
\meeq{
	Q_{0,0}^{(1,1,0)}(x,y) = Q_{0,0}^{(1,0,0)}(x,y), \ccr
	2n Q_{n,0}^{(1,1,0)}(x,y) = (n+1)Q_{n,0}^{(1,0,0)}(x,y)- n Q_{n-1,0}^{(1,0,0)}(x,y), \ccr	
	4n Q_{n,k}^{(1,1,0)}(x,y) = (n+k+1)Q_{n,k}^{(1,0,0)}(x,y)-(n-k)Q_{n,k-1}^{(1,0,0)}(x,y)\\
	&+(k-n)Q_{n-1,k}^{(1,0,0)}(x,y)+(n+k-1)Q_{n-1,k-1}^{(1,0,0)}(x,y),\ccr
	2Q_{n,n}^{(1,1,0)}(x,y) = Q_{n,n}^{(1,0,0)}(x,y)-Q_{n,n-1}^{(1,0,0)}(x,y)+Q_{n-1,n-1}^{(1,0,0)}(x,y). 
}	
\meeq{
	2 Q_{n,0}^{(1,1,0)}(x,y) = Q_{n,0}^{(0,1,0)}(x,y)+ Q_{n-1,0}^{(0,1,0)}(x,y), \ccr	
	2n Q_{n,k}^{(1,1,0)}(x,y) = (n-k)\br[Q_{n,k}^{(0,1,0)}(x,y)+Q_{n-1,k}^{(0,1,0)}(x,y)],\ccr
	Q_{n,n}^{(1,1,0)}(x,y) = Q_{n,n}^{(0,1,0)}(x,y). 
	}
\meeq{
	Q_{0,0}^{(1,0,1)}(x,y) = Q_{0,0}^{(1,0,0)}(x,y), \ccr
	2n Q_{n,0}^{(1,0,1)}(x,y) = (n+1)Q_{n,0}^{(1,0,0)}(x,y)- n Q_{n-1,0}^{(1,0,0)}(x,y), \ccr	
	4n Q_{n,k}^{(1,0,1)}(x,y) = -(n+k+1)Q_{n,k}^{(1,0,0)}(x,y)-(n-k)Q_{n,k-1}^{(1,0,0)}(x,y)\\
	&+(n-k)Q_{n-1,k}^{(1,0,0)}(x,y)+(n+k-1)Q_{n-1,k-1}^{(1,0,0)}(x,y),\ccr
	2Q_{n,n}^{(1,0,1)}(x,y) = -Q_{n,n}^{(1,0,0)}(x,y)-Q_{n,n-1}^{(1,0,0)}(x,y)+Q_{n-1,n-1}^{(1,0,0)}(x,y). \cr
}
\meeq{
	2 Q_{n,0}^{(1,0,1)}(x,y) = Q_{n,0}^{(0,0,1)}(x,y)+ Q_{n-1,0}^{(0,0,1)}(x,y), \ccr	
	2n Q_{n,k}^{(1,0,1)}(x,y) = (n-k)\br[Q_{n,k}^{(0,0,1)}(x,y)+Q_{n-1,k}^{(0,0,1)}(x,y)],\ccr
	Q_{n,n}^{(1,0,1)}(x,y) = Q_{n,n}^{(0,0,1)}(x,y). \cr
}
\meeq{
	2 Q_{n,0}^{(0,1,1)}(x,y) = -Q_{n,0}^{(0,1,0)}(x,y)+ Q_{n-1,0}^{(0,1,0)}(x,y), \ccr	
	2n Q_{n,1}^{(0,1,1)}(x,y) = -2(n+1)Q_{n,1}^{(0,1,0)}(x,y)-n Q_{n,0}^{(0,1,0)}(x,y)\\
	&+ 2(n-1)Q_{n-1,1}^{(0,1,0)}(x,y) +n Q_{n-1,0}^{(0,1,0)}(x,y),\ccr	
	2n(2k-1) Q_{n,k}^{(0,1,1)}(x,y) = \\
	-(k-1)(n+k)&Q_{n,k}^{(0,1,0)}(x,y)-(k-1)(n-k+1)Q_{n,k-1}^{(0,1,0)}(x,y)\\
	+(k-1)(n-k)&Q_{n-1,k}^{(0,1,0)}(x,y)+(k-1)(n+k-1)Q_{n-1,k-1}^{(0,1,0)}(x,y). \cr
}
\meeq{
	2 Q_{n,0}^{(0,1,1)}(x,y) = -Q_{n,0}^{(0,0,1)}(x,y)+ Q_{n-1,0}^{(0,0,1)}(x,y), \ccr	
	2n Q_{n,1}^{(0,1,1)}(x,y) = 2(n+1)Q_{n,1}^{(0,1,0)}(x,y)+n Q_{n,0}^{(0,0,1)}(x,y)\\
	&-2(n-1)Q_{n-1,1}^{(0,1,0)}(x,y) -n Q_{n-1,0}^{(0,0,1)}(x,y),\ccr	
	2n(2k-1) Q_{n,k}^{(0,1,1)}(x,y) = \\
	(k-1)(n+k)Q_{n,k}^{(0,0,1)}(x,y)&-(k-1)(n-k+1)Q_{n,k-1}^{(0,0,1)}(x,y)\\
	-(k-1)(n-k)Q_{n-1,k}^{(0,0,1)}(x,y)&+(k-1)(n+k-1)Q_{n-1,k-1}^{(0,0,1)}(x,y). 	
}
\end{corollary}
\begin{proof}
These are either immediate from definitions or are obtained by rearranging recurrence relationships found in \cref{Corollary:Lowering}. 
\end{proof}

The three-edge Dirichlet basis also satisfy several sparse recurrence relationships. 
\begin{corollary}\label{Corollary:ThreeEdgeConversion}
The following recurrence relationships hold: 
\meeq{
	Q_{0,0}^{(1,1,1)}(x,y) = Q_{0,0}^{(0,1,1)}(x,y),  \ccr
	Q_{1,0}^{(1,1,1)}(x,y) = 2Q_{1,0}^{(0,1,1)}(x,y) -Q_{0,0}^{(0,1,1)}(x,y), \ccr
	Q_{1,1}^{(1,1,1)}(x,y) = Q_{1,1}^{(0,1,1)}(x,y), \ccr
	(2n-1)Q_{n,0}^{(1,1,1)}(x,y) = (n-1)\br[Q_{n,0}^{(0,1,1)}(x,y) +Q_{n-1,0}^{(0,1,1)}(x,y)], \ccr
	(2n-1)Q_{n,k}^{(1,1,1)}(x,y) = (n-k)\br[Q_{n,k}^{(0,1,1)}(x,y) +Q_{n-1,k}^{(0,1,1)}(x,y)], \ccr
	Q_{n,n}^{(1,1,1)}(x,y) = Q_{n,n}^{(0,1,1)}(x,y). \cr\ccr
	Q_{0,0}^{(1,1,1)}(x,y) = Q_{0,0}^{(1,0,1)}(x,y),  \ccr
	Q_{1,0}^{(1,1,1)}(x,y) = -2Q_{1,0}^{(1,0,1)}(x,y) +Q_{0,0}^{(1,0,1)}(x,y), \ccr
	Q_{1,1}^{(1,1,1)}(x,y) = 2Q_{1,1}^{(1,0,1)}(x,y)+Q_{1,0}^{(1,0,1)}(x,y)-Q_{0,0}^{(1,0,1)}(x,y), \ccr
	(2n-1)Q_{n,0}^{(1,1,1)}(x,y) = (n-1)\br[-Q_{n,0}^{(1,0,1)}(x,y) +Q_{n-1,0}^{(1,0,1)}(x,y)], \ccr
	(2n-1)Q_{n,1}^{(1,1,1)}(x,y) = 2 (n+1)Q_{n,1}^{(1,0,1)}(x,y) +(n-1) Q_{n,0}^{(1,0,1)}(x,y) \\
	&-2(n-1) Q_{n-1,1}^{(1,0,1)}(x,y) -(n-1)Q_{n-1,0}^{(1,0,1)}(x,y)], \ccr
	(2n-1)(2k-1)Q_{n,k}^{(1,1,1)}(x,y) = \\
	(n+k)(k-1)&Q_{n,k}^{(1,0,1)}(x,y) -(n-k)(k-1)Q_{n,k-1}^{(1,0,1)}(x,y)\\
	-(n-k)(k-1) &Q_{n-1,k}^{(1,0,1)}(x,y) +(n+k-2)(k-1)Q_{n-1,k-1}^{(1,0,1)}(x,y), \ccr
	(2n-1)Q_{n,n}^{(1,1,1)}(x,y) =\\
	 (n-1)&\br[Q_{n,n}^{(1,0,1)}(x,y)-Q_{n,n-1}^{(1,0,1)}(x,y)+Q_{n-1,n-1}^{(1,0,1)}(x,y)]. \cr
	}
\meeq{
	Q_{0,0}^{(1,1,1)}(x,y) = Q_{0,0}^{(1,1,0)}(x,y),  \ccr
	Q_{1,0}^{(1,1,1)}(x,y) = -2Q_{1,0}^{(1,1,0)}(x,y) +Q_{0,0}^{(1,1,0)}(x,y), \ccr
	Q_{1,1}^{(1,1,1)}(x,y) = -2Q_{1,1}^{(1,1,0)}(x,y)-Q_{1,0}^{(1,1,0)}(x,y)+Q_{0,0}^{(1,1,0)}(x,y), \ccr
	(2n-1)Q_{n,0}^{(1,1,1)}(x,y) = (n-1)\br[-Q_{n,0}^{(1,1,0)}(x,y) +Q_{n-1,0}^{(1,1,0)}(x,y)], \ccr
	(2n-1)Q_{n,1}^{(1,1,1)}(x,y) = -2 (n+1)Q_{n,1}^{(1,1,0)}(x,y) -(n-1) Q_{n,0}^{(1,1,0)}(x,y) \\
	+2(n-1) &Q_{n-1,1}^{(1,1,0)}(x,y) +(n-1)Q_{n-1,0}^{(1,1,0)}(x,y)], \ccr
	(2n-1)(2k-1)Q_{n,k}^{(1,1,1)}(x,y) =\\
	 -(n+k)(k-1) &Q_{n,k}^{(1,1,0)}(x,y) -(n-k)(k-1)Q_{n,k-1}^{(1,1,0)}(x,y) \\
	 +(n-k)(k-1)&Q_{n-1,k}^{(1,1,0)}(x,y) +(n+k-2)(k-1)Q_{n-1,k-1}^{(1,1,0)}(x,y), \ccr
	(2n-1)Q_{n,n}^{(1,1,1)}(x,y) = \\
	 (n-1) &\br[-Q_{n,n}^{(1,1,0)}(x,y)-Q_{n,n-1}^{(1,1,0)}(x,y)+Q_{n-1,n-1}^{(1,1,0)}(x,y)]. \cr		
}
\end{corollary}
\begin{proof}
These are either immediate from definitions or are obtained by rearranging recurrence relationships found in \cref{Corollary:Lowering}. 
\end{proof}

\subsection{Recurrence relationships for the partial derivatives of the Dirichlet basis}
We now turn to recurrence relationships for the partial derivatives of the Dirichlet basis, which are needed when imposing Neumann boundary conditions. 

\begin{corollary}\label{Corollary:TwoEdgeDerivative}
The following recurrence relationships hold: 
\meeq{
	{\partial \over \partial y} Q_{n,0}^{(0,1,1)}(x,y) = 0, \ccr
	{\partial \over \partial y} Q_{n,1}^{(0,1,1)}(x,y) = -2 P_{n-1,0}(x,y), \ccr
	{\partial \over \partial y} Q_{n,k}^{(0,1,1)}(x,y) = (1-k) P_{n-1,k-1}(x,y), \ccr
	{\partial \over \partial x} Q_{n,0}^{(1,0,1)}(x,y) = n P_{n-1,0}(x,y), \ccr
	{\partial \over \partial x} Q_{n,k}^{(1,0,1)}(x,y) = {k-n \over 2} \br[{P_{n-1,k-1}(x,y) +P_{n-1,k}(x,y)}],\ccr
	{\partial \over \partial x} Q_{n,n}^{(1,0,1)}(x,y) = -n P_{n-1,n-1}(x,y), \ccr
	{\partial \over \partial z} Q_{n,0}^{(1,1,0)}(x,y) = -n P_{n-1,0}(x,y),   \ccr
	{\partial \over \partial z} Q_{n,k}^{(1,1,0)}(x,y) = {n-k \over 2} \br[{P_{n-1,k-1}(x,y) - P_{n-1,k}(x,y)}],\ccr
	{\partial \over \partial z} Q_{n,n}^{(1,1,0)}(x,y) = n P_{n-1,n-1}(x,y).	
	}	
\end{corollary}
\begin{proof}
The first three relations follow from the weighted partial differentiation relationships (see \cref{Corollary:WeightedDifferentiation}). The fourth relation requires the additional property that
$$
	{\D \over \dx} \br[x \tilde P_{n-1}^{(0,1)}(x)] = n \tilde P_{n-1}^{(1,0)}(x),
$$
which follows from~\cite[15.5.6]{DLMF}. The fifth relationship also follows from \cref{Corollary:WeightedDifferentiation}. For the sixth equation, if we define $t = y/(1-x)$, then the relation reduces to
$$
(1-x)^{n-1} \left[ ((n-1)(1-t) + 1) \tilde P_{n-1}^{(1,0)}(t) - t(1-t) {\D \over \D t} \tilde P_{n-1}^{(1,0)}(t)  \right]= n (1-x)^{n-1} \tilde P_{n-1}(t),
$$
and this expression follows from ${\cal L}_2^\dagger$ in~\cite[Lem.~1]{Olver_18_01}. The last three relations follow from the same manipulation.
\end{proof}

\end{document}